\documentclass[a4paper]{article}

\usepackage[utf8]{inputenc}
\usepackage{csquotes}

\usepackage[a4paper, top=3.5cm, bottom=3.5cm, left=3.5cm, right=3.5cm]{geometry}

\usepackage{amsmath}
\usepackage{amsfonts}

\usepackage{graphicx}
\usepackage{subcaption}

\usepackage{amsmath}
\numberwithin{equation}{section}
\allowdisplaybreaks
\usepackage{amssymb}
\usepackage{commath}
\usepackage{mathtools}
\usepackage{bbm}
\usepackage{nicefrac}

\usepackage{siunitx}

\usepackage{amsthm}
\theoremstyle{plain}
  \newtheorem{theorem}{Theorem}
  \newtheorem{lemma}[theorem]{Lemma}

\theoremstyle{definition}

\numberwithin{theorem}{section}

\usepackage[%
  backend=bibtex,bibencoding=ascii,
  style=numeric-comp,
  giveninits=true, uniquename=init, 
  natbib=true,
  url=true,
  doi=true,
  isbn=false,
  backref=false,
  maxnames=99,
  ]{biblatex}
\usepackage{booktabs}
\usepackage{multirow}

\usepackage[plainpages=false,pdfpagelabels,hidelinks,unicode]{hyperref}

\begingroup\expandafter\expandafter\expandafter\endgroup
\expandafter\ifx\csname pdfsuppresswarningpagegroup\endcsname\relax
\else
\fi

\newcommand{\R}{\mathbb{R}}
\newcommand{\CN}{\mathbb{C}}

\newcommand{\dt}{{\Delta t}}

\newcommand{\I}{\operatorname{I}}

\newcommand{\ut}{\tilde{u}}
\newcommand{\bt}{\tilde{b}}
\newcommand{\pt}{{\tilde{p}}}
\newcommand{\perturb}{perturb}

\usepackage[boxed]{algorithm}
\usepackage{algpseudocode}

\usepackage{tikz}
\usetikzlibrary{shapes.geometric, arrows}
\usetikzlibrary{calc}
\usepackage{color}

\newcommand{\LPblock}[1]{
\tikzstyle{textbox} = [draw=black, fill=white, very thick,
    rectangle, rounded corners, inner sep=10pt, inner ysep=20pt]
\tikzstyle{titlebox} =[fill=white,draw=black, text=black]
\begin{center}
\begin{tikzpicture}
\node [textbox] (box){%
\begin{minipage}{0.8\textwidth}
\vspace{-0.3cm}
{#1}
\vspace{-0.3cm}
\end{minipage}
};
\node[titlebox, right=10pt] at (box.north west) {LP};
\end{tikzpicture}
\end{center}
}

\newcommand{\orcid}[1]{ORCID:~\href{https://orcid.org/#1}{#1}}
\usepackage{authblk}

\newenvironment{keywords}{\par\textbf{Key words.}}{\par}
\newenvironment{AMS}{\par\textbf{AMS subject classification.}}{\par}

\title{Positivity-Preserving Adaptive Runge--Kutta Methods}
\author[1]{Stephan Nüßlein\thanks{\orcid{0000-0002-2455-4222}}}
\author[1]{Hendrik Ranocha\thanks{\orcid{0000-0002-3456-2277}}}
\author[1]{David I. Ketcheson\thanks{\orcid{0000-0002-1212-126X}}}
\affil[1]{%
King Abdullah University of Science and Technology (KAUST),
Computer Electrical and Mathematical Science and Engineering Division (CEMSE),
Thuwal, 23955-6900, Saudi Arabia}

\makeatletter
\hypersetup{pdfauthor={Stephan Nüßlein, Hendrik Ranocha, David I. Ketcheson}}
\hypersetup{pdftitle={\@title}}
\makeatother

\begin{document}

\maketitle

\begin{abstract}
Many important differential equations model quantities whose value
must remain positive or stay in some bounded interval.
These bounds may not be preserved when the model is solved numerically.
We propose to ensure positivity or other bounds by applying Runge--Kutta
integration in which the method weights are adapted in order to
enforce the bounds.  The weights are chosen at each step after calculating the
stage derivatives, in a way that also preserves (when possible) the order of
accuracy of the method.  The choice of weights is given by the solution
of a linear program.
We investigate different approaches to choosing the weights by considering
adding further constraints.
We also provide some analysis of the properties
of Runge--Kutta methods with perturbed weights.  Numerical examples demonstrate
the effectiveness of the approach, including application to both stiff and
non-stiff problems.
\end{abstract}



\section{Introduction}

Many physical processes can be described with differential equations.
The physical quantities that are involved in these processes often only make sense if they remain within certain bounds.
For instance, concentrations must be non-negative (we will often say simply {\em positive} for short), while
probabilities or mass fractions must remain in $[0,1]$.
The ordinary differential equations (ODEs) or partial differential equations
(PDEs) that model these quantities are often too complex to be solved
analytically and therefore require numerical approximation.
Numerical methods generally may not satisfy these bound constraints.
In the present work, we develop an approach to ensuring positivity
or other bound constraints using Runge--Kutta methods (RKMs) for the
solution of ODEs or semi-discretized PDEs.

We say an initial value problem
\begin{align} \label{ode}
    u'(t) & = f(t,u) \\
    u(0) & = u_0
\end{align}
where $u\colon [0,T] \to \R^m$ is positive if
\begin{align} \label{continuous-positivity}
    u(0)\ge 0 \implies u(t) \ge 0 \text{ for all } t \in [0,T].
\end{align}
Here and in the following, inequalities like $u \geq 0$ are meant componentwise.
A sufficient condition for positivity of \eqref{ode} is
\begin{equation} \label{eq:condition_ODE_pos}
u_i=0 \implies f_i(t,[u_1,\cdots,u_i,\cdots,u_n]^T) \geq 0 \quad \forall {u \geq 0,}\; \ \  \forall {t\in[0,T]}.
\end{equation}
For such ODEs, the backward Euler method is guaranteed to preserve positivity
under any step size, while the forward Euler will preserve positivity for small enough $\dt$ \cite{hundsdorfer_numerical_2003}.
Any RKM (or in fact any general linear method) that is unconditionally positivity preserving for all positive ODEs
must have order $\le 1$ \cite{bolley_conservation_1978}.
For any higher order method, we expect positivity only under some restriction
on the time step size.

Several approaches to ensuring numerical positivity exist in the literature.
The most basic approach is orthogonal projection onto the positive orthant,
which means simply setting negative values to zero \cite{shampine1986conservation}.  This approach is often problematic;
for instance, it will violate linear invariants such as mass conservation.
As another approach, one may use event finding methods in order to stop when any solution component
reaches zero, and then proceed in some special way
\cite{shampine_non-negative_2005}. This approach is implemented in the MATLAB
ODE Suite along with the idea of redefining the ODE outside the positive orthant (usually
by evaluating at the nearest point on the boundary of the positive orthant).
If positivity is preserved under a forward Euler step (with
some step size restriction $\dt \le \dt_\text{FE}$), then any strong stability preserving Runge--Kutta (SSPRK)
method will also preserve positivity (with a modified step size restriction)~\cite{gottlieb_strong_2011}.
Specifically, the positivity of the method is ensured for time steps
$\dt \leq {\mathcal C} \dt_\text{FE}$, where ${\mathcal C}$ depends on the SSPRK method.
Modified Patankar--Runge--Kutta (MPRK) methods represent another approach to ensuring
positivity for specific classes of ODEs. MPRK methods introduce multiplicative
factors within the Runge--Kutta stages to ensure positivity, but require the solution
of a linear algebraic system; see e.g.\ \cite{kopecz_comparison_2019} and references therein.
Finally, we mention diagonally split Runge--Kutta (DSRK) methods, which can be unconditionally
positive and have order higher than one. Like MPRK schemes, DSRK methods avoid
the restriction mentioned above because
they are not general linear methods \cite{horvath_positivity_1998}.  However, in practice
unconditionally positive DSRK methods are less accurate than backward Euler for
large step sizes \cite{macdonald2007}.

The rather discouraging theoretical result of \cite{bolley_conservation_1978} shows that one
should not hope to preserve positivity with a single method for every problem and every
initial condition.  In the present work we take an approach based on the idea that for
a particular problem and initial condition, there often exists a method of high order
that is positivity preserving, at least for a single step.
The main idea is to adaptively choose the weights $b$ of the RKM, after
the stage values are known, in a way that ensures positivity.
The selection of the weights requires the solution of a linear program (LP) at every
step for which the numerical solution would otherwise be non-positive.
This is a significant cost, but may in some cases be an economical alternative to rejecting
a step or using excessively small step sizes.

The idea of using different weights within an RKM is not new; for instance it is
the basis of error approximation using embedded RK pairs \cite{hairer_solving_1993}.
The idea of adapting the weights after calculating the stage values has also been used,
for instance in \cite{ketcheson_spatially_2013}.  In this case it is used to
adapt the properties of the time integrator for a method of lines solution of a
PDE.  Another class of methods that adapt the weights at the end of an RK step
are the relaxation Runge--Kutta (RRK) methods.  In these, the weights are
scaled by a scalar relaxation parameter in order to guarantee conservation or
monotonicity of a desired functional; e.g.\ to conserve or dissipate energy or
entropy
\cite{ketcheson_relaxation_2019,ranocha_relaxation_2019,ranocha2020general}.

Our means to ensure positivity can be interpreted as a projection
approach, where the numerical solution is adapted to satisfy the
positivity constraint at the end of each step. In contrast to
simple orthogonal projection, which has also been proposed to deal
with positivity constraints \cite{shampine1986conservation},
our approach preserves all linear invariants of the given ODE. These
invariants can be very important, e.g.\ the total mass for a transport
problem or in reaction systems. Preservation of linear
invariants has been shown to be an important advantage of RRK methods
over orthogonal projection methods
\cite{ranocha2020relaxationHamiltonian}. Of course, it is also
possible to enforce the preservation of linear invariants in projection
methods, but the invariants have to be known explicitly
\cite{sandu2001positive}.

The paper unfolds as follows. In Section~\ref{sec:main_idea} the main idea is
explained. Section~\ref{sec:LP} contains the formulation of the linear program
for selection of the weights at each step.
Section~\ref{sec:integration} describes how the new approach can be used with
different RKMs, how it can be combined with adaptive error control, and how the
region of absolute stability can be approximated.
In Section~\ref{sec:Numeric_Results} numerical results are given for multiple test problems.
A conclusion is given in Section~\ref{sec:conclusion}.

\section{Bound-preserving adaptive Runge--Kutta methods}\label{sec:main_idea}

When computing the solution of an ODE $u ' = f(t,u) $ using an RKM with $s$ stages and the Butcher tableau
\begin{align}
\renewcommand{\arraystretch}{1.2}
\begin{array}{c|c}
c &  A \\
\hline
 & b^T\\
\end{array}
\end{align}
the stage values are computed according to
\begin{equation}\label{eq:stagevalues}
y_j =  u^n + \dt \sum_{k = 1}^{s} a_{jk} f(t^n + \dt c_k,y_k),  \quad j = 1,\cdots,s.
\end{equation}
Based on these values, the next solution $u^{n+1}$ is computed as
\begin{equation} \label{eq:rkstep}
u^{n+1} = u^n + \dt \sum_{j  = 1}^s f(t^n + \dt c_j,y_j) b_j .
\end{equation}
Let $f_j = f(t^n + \dt c_j,y_j)$; then we can write \eqref{eq:rkstep} as
\begin{equation}\label{eq:Combination}
u^{n+1} = u^n + \dt F b,
\end{equation}
where the $j$th column of $F$ is $f_j$.
We wish to impose the discrete analog of \eqref{continuous-positivity}; i.e.\
\begin{align} \label{positivity}
    u^n\ge 0 \implies u^{n+1} \ge 0,
\end{align}
or more general bound constraints
\begin{align}
    \alpha \le u^n\le \beta \implies \alpha \le u^{n+1} \le \beta.
\end{align}
We will focus on the case of positivity while keeping in mind that
the methodology extends to general bounds.
The main idea of the present work is that if the new solution $u^{n+1}$ contains
negative entries, we can replace the weights in \eqref{eq:Combination} with
a set of modified weights $\bt$ such that the resulting solution is positive:
\begin{equation}\label{eq:ut}
\ut^{n+1} = u^n + \dt F \bt^n \ge 0.
\end{equation}
Indeed, we can view \eqref{eq:ut} as a linear constraint on the choice of
modified weights $\bt^n$.  Since we have already computed the intermediate stages,
$F$ is a known, fixed matrix.  In order to ensure that the modified solution $\ut^{n+1}$
is accurate, we can also constrain $\bt$ to satisfy the Runge--Kutta order conditions
up to some order (ideally, the same order as the original method).  Observe that
all of the order conditions are linear in the weights, so that these additional constraints
take the form
$$
Q\bt = r
$$
for some fixed matrix $Q$ and vector $r$.
By applying this technique at each step, we integrate \eqref{ode} with a
sequence of Runge--Kutta methods with coefficients $(A,\bt^n)$.  At any step
for which the solution $u^{n+1}$ produced by method $(A,b)$ is positive, we do not
need to modify the weights and can simply accept this unmodified solution.
Note that linear invariants (such as mass conservation) of the solution are
automatically preserved in this approach, since at each step we use a Runge--Kutta
method.

\subsection{Example I}\label{sec:example_lin}

The main goal is to choose a method $(A,b)$ such that $u^{n+1}$
approximates the solution of the ODE $u(t_{n+1})$.
An obvious objective while modifying the Runge--Kutta coefficients
is to retain a high order of accuracy, but this does not fully
determine the choice of weights in general.
To get a better understanding for the method we consider the
behavior for a simple problem.

We take the linear, positivity preserving ODE
\cite{kopecz_unconditionally_2018}
\begin{equation}
  u'(t) = L u(t),
  \quad
  u(0) = \begin{pmatrix} 1 \\ 0 \end{pmatrix},
  \qquad
  L = \begin{pmatrix} -5 & 1 \\ 5 & -1\end{pmatrix},
\end{equation}
and use the three stage, third order SSP method SSP(3,3)
\begin{align}
\begin{array}{c|ccc}
0 &  &  & \\
1 & 1 &  & \\
\nicefrac{1}{2} & \nicefrac{1}{4} & \nicefrac{1}{4} & \\
\hline
 & \nicefrac{1}{6} & \nicefrac{1}{6} & \nicefrac{2}{3}\\
\end{array}
\end{align}
of \cite{shu1988efficient}.
The matrix $L$ has the eigenvalues zero and $-6$ and its operator
norm is $2 \sqrt{13}$.
The real-axis stability interval of SSP33 includes the interval $[-2.5,0]$.
We take $\dt = \nicefrac{1}{3}$, which satisfies the spectral
condition and guarantees boundedness (though not monotonicity)
of the solution.
The corresponding stage derivatives are
\begin{equation}
  f(y_1) = \begin{pmatrix} -5 \\ 5 \end{pmatrix},\quad
  f(y_2) = \begin{pmatrix} 5 \\ -5 \end{pmatrix},\quad
  f(y_3) = \begin{pmatrix} -5 \\ 5 \end{pmatrix}.
\end{equation}
The value of the next step using the standard weights is
\begin{equation}
  u^1 = \begin{pmatrix} \nicefrac{-1}{9} \\ \nicefrac{10}{9} \end{pmatrix}.
\end{equation}
Since the first component of the new solution is negative,
we want to adapt the weights to ensure positivity.
All weights that comply with the constraints for first and second
order of accuracy can be expressed as
\begin{equation}
  \bt =
  \begin{pmatrix}
    \nicefrac{1}{6} \\
    \nicefrac{1}{6} \\
    \nicefrac{2}{3}
  \end{pmatrix}
  + \alpha \begin{pmatrix}
    \nicefrac{1}{2} \\
    \nicefrac{1}{2} \\
    -1
  \end{pmatrix},
  \qquad
  \alpha \in \R.
\end{equation}
We have one degree of freedom for the choice of the weights,
parameterized by $\alpha$.
If the general expression for the weights is inserted in
\eqref{eq:Combination} the general solution is
\begin{equation}
  u^{1}
  =
  u^0 + \dt \left(f_1, f_2, f_3\right) \bt
  =
  \begin{pmatrix}
    \nicefrac{-1}{9} \\
    \nicefrac{10}{9}
  \end{pmatrix}
  +\alpha \begin{pmatrix}
    5 \\
    -5
  \end{pmatrix}.
\end{equation}
By changing the parameter $\alpha$, the weights and the new
solution are altered. With a suitable choice of
$\alpha \in \left[ \nicefrac{1}{45}, \nicefrac{2}{9} \right]$,
any $u$ that complies with mass conservation and positivity can
be reached.
By adding additional constraints on the weights, the choice of
$\alpha$ can be narrowed down.
An objective function is also needed to make the choice unique. This should be
designed in a way to prefer weights that are close to the original weights.

We see that the choice of $\bt$ is subject to linear equality and inequality
constraints.  If we choose a linear objective function, the resulting problem
for finding the modified weights is a linear program, which can be efficiently
solved by standard algorithms.  A natural choice of objective function is
$$
\text{minimize } \|\bt - b\|_1.
$$
The resulting problem can be phrased as an LP by using slack variables.
In general, this LP may not have a solution; we can relax the constraints
by requiring a lower order of consistency than the design order of the
method.  These choices and alternatives will be considered in Section~\ref{sec:LP}.

In contrast to other projection methods
\cite{shampine1986conservation,sandu2001positive},
minimizing the deviation of the weights instead of the deviation
of the projected solution is computationally much more efficient
for large systems, arising for example in the discretization of
PDEs.

\subsection{Example II}\label{sec:example_reac}

To illustrate the usage of the method we consider the reaction system
\cite{kopecz_comparison_2019}
\begin{subequations}
\label{eq:Reaction}
\begin{align}
u_1' &= 0.01u_2 + 0.01 u_3 +0.003u_4 - \frac{u_1 u_2}{0.01+u_1}, \\
u_2' &= \frac{u_1u_2}{0.01+u_1}-0.01 u_2-0.5(1-\exp(-1.21 u_2^2)) u_3 -0.05 u_2, \\
u_3' &= 0.5(1-\exp(-1.21u_2^2)) u_3 - 0.01 u_3 -0.02 u_3, \\
u_4' &=0.05 u_2 + 0.02 u_3 + 0.003u_4,
\end{align}
\end{subequations}
with initial conditions
\begin{equation}
u(0) = (8,2,1,4)^T.
\end{equation}
Note that we wrote \eqref{eq:Reaction} as in
\cite{kopecz_comparison_2019}, sometimes using multiple terms containing the same
variables but with different constants, e.g.\ $-0.01 u_2 -0.05 u_2$ in the
time derivative of $u_2$. This notation is useful to see the structure of
a production-destruction system which is exploited for positivity-preserving
(modified) Patankar--Runge--Kutta methods as in \cite{kopecz_comparison_2019}.
We will use the same notation also later in this article.

Using the Cash--Karp RK5 method \cite{cash1990variable} and $\dt = 0.005$ to
solve \eqref{eq:Reaction}, the approximated solution contains negative values.
This causes qualitatively wrong solutions to the problem.
In Figure\,\ref{fig:exampleI} the obtained results are plotted with dashed lines.
At $t=1.905$ the value of $u_1$ gets negative. This leads to a diverging solution.

Now the weights are adapted. The adapted weights are of 4th order. The results are also plotted in Figure\,\ref{fig:exampleI}, with solid lines.
The positivity constraint is now fulfilled. A qualitatively correct solution is obtained.

\begin{figure}[ht]
    \centering
    \includegraphics[width=0.75\textwidth]{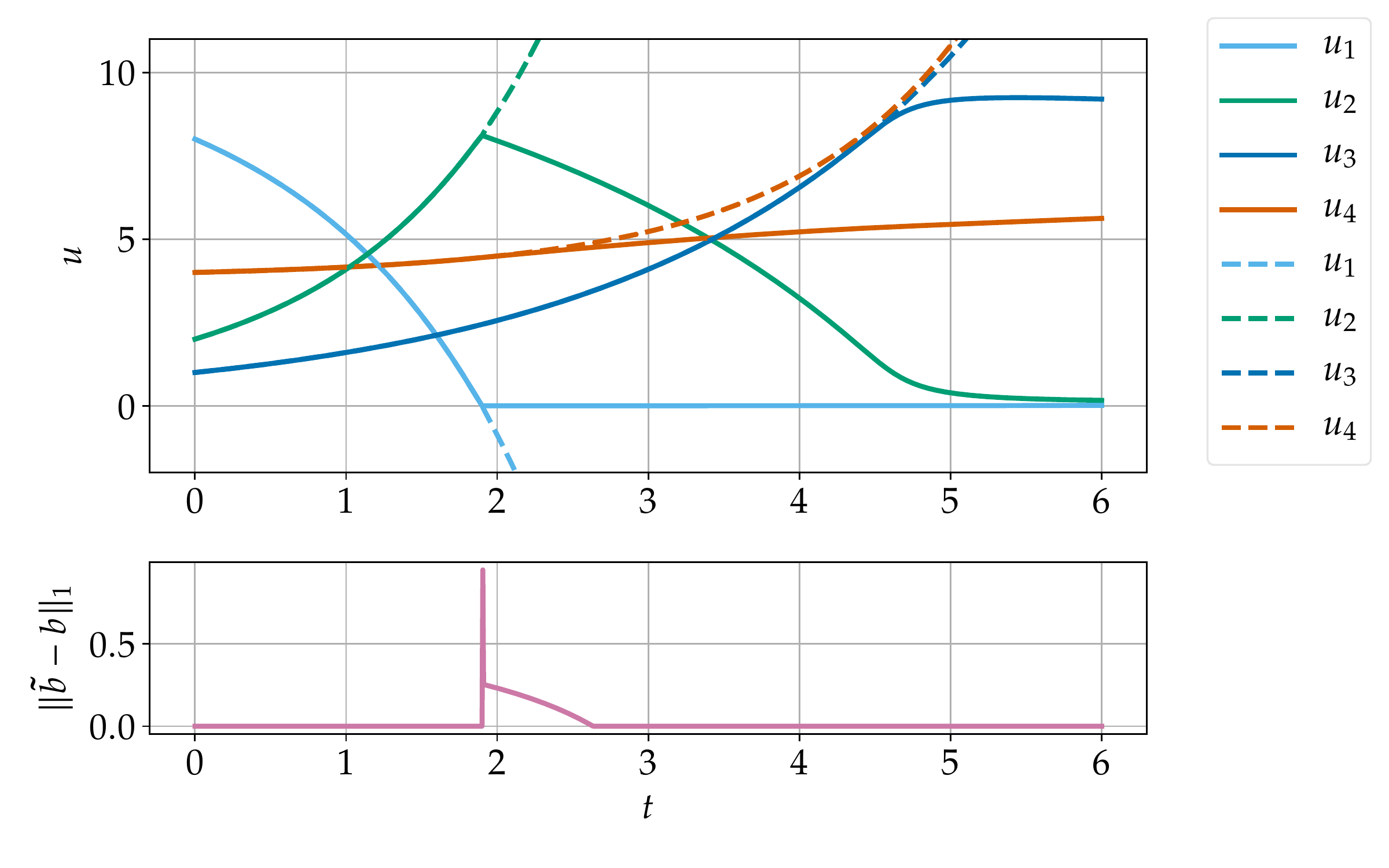}
    \caption{Numerical approximation of the reaction problem \eqref{eq:Reaction} computed with Cash--Karp RK5 and $\dt = 0.005$. The dashed lines are the approximations obtained without the adaption of the weights. The lower plot shows the adaptation of the weights. }
    \label{fig:exampleI}
\end{figure}

The difference $\|\bt-b\|_1$ is also plotted in Figure\,\ref{fig:exampleI}.
No modification of the weights is required for $t<1.905$. At $t=1.905$ the
weights are first adapted to ensure the positivity of the solution. For
$t>2.63$ the original set of weights again lead to a positive solution,
and no further modification is necessary.

\section{Selection of modified weights}\label{sec:LP}

In this section, we consider further the formulation of the LP to choose
the modified weights $\bt$.  In particular, we focus on the choice
of objective function and how to relax the constraints to ensure that
a feasible solution exists.

\subsection{Order conditions}\label{sec:OrderCond}

The order conditions for an $s$-stage, order $p$ RKM are a set of equations depending on $A$, $b$, and
$c$.  As mentioned already, if $A$ and $c$ are given, the order conditions are
linear in $b$ and can be written as $Q_p b = r_p$, where
$Q_p\in{\mathbb R}^{v\times s}, r_p\in{\mathbb R}^v$ represent the set of all conditions up to
and including order $p$. Here $v$ is the number of order conditions. It may not be possible to find modified weights
that also satisfy the conditions of order $p$ and yield positivity, so
in general the modified weights will be a solution of
$$
  Q_\pt \bt = r_\pt
$$
for some $\pt \le p$.   Since we have $s$ degrees of freedom $\bt_j$,  we need
at a minimum to choose $\pt$ so that $\mathrm{rank}(Q_\pt) < s$.
Because the quadrature conditions are linearly independent, we have $\mathrm{rank}(Q_p)\ge p$,
so we must take $\pt \le s$.  In general we may need to take $\pt$ even smaller
in order to achieve positivity.

\subsection{Choice of objective function and additional constraints}
In the design of Runge--Kutta methods, weights are carefully chosen
not only to satisfy the order conditions but also
to give desirable properties such as a good region of absolute
stability, small error coefficients, and so forth.
Replacing these carefully-chosen weights $b$ with arbitrary weights
$\bt$ could lead to the loss of these desirable properties.
In order to preserve as much as possible the good properties of the method, we use as
objective function $\|\bt - b\|_1$.  This has the additional benefit
of penalizing weights with large magnitude in general, avoiding
large truncation or cancellation errors.  This also ensures that if
no negative solution values appear, the solution of the LP is
simply the original method weights.  Thus we have the following LP:

\LPblock{
(Free adaptation) Given $F$, $\pt$, and $b$, find $\bt$ that minimizes $\|\bt - b\|_1$ subject to
\begin{subequations}
\label{eq:free-adaptation}
\begin{align}
u^{n+1}&=u^n+\dt F \bt \geq 0, \label{eq:direct_pos}\\
Q_\pt\bt &=r_\pt \label{eq:direct_Order}.
\end{align}
\end{subequations}
}

Of course, there is still no guarantee that the modified weights
will be close to the original method weights.  In some examples
we have observed that large modifications of the weights can lead
to inaccurate solutions even though the order conditions are satisfied.
In order to avoid issues that might be caused by poor weights, we can
 additionally use either or both of the following ideas:

\begin{itemize}
    \item Convex adaptation:
      Select in advance a set of desirable weight vectors $b^1, b^2, \dots, b^K$
      corresponding to known good methods,
      and restrict the choice of $\bt$ to convex combinations of this set.
    \item Stepsize control:
      Require that the perturbation $\|\tilde u - u\|$ is small and reject
      the step if it is not.
\end{itemize}

We discuss the first idea here; the second is deferred to section\,\ref{sec:error}.
Ideally every element of the set of potential weight vectors would correspond to
a method of the same order as the original method.   Due to linearity of the
order conditions, any linear combination of such weights would also yield a
method of the same order.  On the other hand,
it is natural to include a weight vector corresponding to the forward Euler method
(for explicit methods) or backward Euler method (for implicit methods), since
these two methods guarantee positivity (unconditionally for backward Euler
and conditionally for forward Euler).
We can formulate an LP using the approach of convex adaptation as follows.
Let $B$ denote the matrix with columns $b^1, b^2, \dots, b^K$ and let
$g\in \R^K$.  The LP is then as follows:

\LPblock{
(Convex adaptation) Given $F$, $B$, and $b$, find $g$ that minimizes $\|\bt - b\|_1$ subject to
\begin{subequations}
\begin{align}
\bt & = Bg, \\
0 & \le g_k \le 1, \\
\sum_{k=1}^K g_k &= 1, \\
u^{n+1}&=u^n+\dt F \bt \geq 0.
\end{align}
\end{subequations}
}

Note that we do not need to impose the order conditions here, since they will be
satisfied by each of the methods and thus (by linearity) by the modified
method.  The order of the modified method will in general be equal to the lowest
order among the component methods.

Both approaches are illustrated in Figure\,\ref{fig:b_space}.

\begin{figure}
    \centering
    \begin{subfigure}[b]{0.45\textwidth}
        \centering
        \begin{tikzpicture}
	\coordinate (borig) at (1.2,1.8);
    \coordinate (b) at (2,1);
	\coordinate (db) at ($(b)-(borig)$);

    \draw [<->,thick] (0,2.5) node (yaxis) [above] {$b_2$}
        |- (3,0) node (xaxis) [right] {$b_1$};

	\draw[thick,red,dotted] ($-0.9*(db)+(borig)$) -- ($0.9*(db)+(b)$);
	\draw[thick,red] ($-0.6*(db)+(borig)$) -- ($0.6*(db)+(b)$);
	
	\draw[thick,->] (borig) -- node[anchor=south west ,pos = 0.5] {$\Delta b$} (b);
	\draw[thick,->] (0,0) -- node[anchor=south east,pos = 0.8] {$b$} (borig);
    \draw[thick,->] (0,0) -- node[anchor=north west,pos = 0.8] {$\tilde{b}$} (b);

\end{tikzpicture}
        \caption{Free adaptation}
        \label{fig:b_direct}
    \end{subfigure}
    \begin{subfigure}[b]{0.45\textwidth}
        \centering
        \begin{tikzpicture}

	\coordinate (bo) at (1.2,1.8);
    \coordinate (b) at (2,1);
    \coordinate (db) at ($(b)-(bo)$);
    \coordinate (b1) at ($0.4*(db)+(b)$);
    \coordinate (b0) at ($-0.2*(db)+(bo)$);;

    \draw [<->,thick] (0,2.5) node (yaxis) [above] {$b_2$}
        |- (3,0) node (xaxis) [right] {$b_1$};

   \node[circle,inner sep=1.5pt,draw,fill,red] (A) at (b0) {};
   \node[circle,inner sep=1.5pt,draw,fill,red] (B) at (b1) {};
   \draw[thick,red,-] (b0) -- (b1);

	\draw[thick,->] (0,0) -- node[anchor=south east,pos = 0.8] {$b^1$} (b0);
    \draw[thick,->] (0,0) -- node[anchor=south east,pos = 0.8] {$\tilde{b}$} (b);
    \draw[thick,->] (0,0) -- node[anchor=north west,pos = 0.8] {$b^2$} (b1);

\end{tikzpicture}
        \caption{Convex adaptation}
        \label{fig:b_convex}
    \end{subfigure}
    \caption{Graphical representation of the two different approaches to adapt the weights for
    a two-stage method.}\label{fig:b_space}
\end{figure}
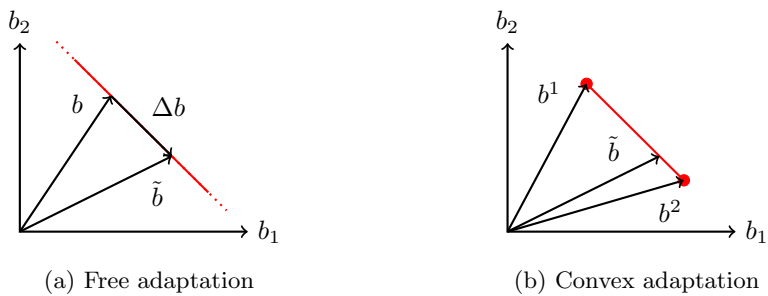

\subsection{Reduction of number of positivity constraints}
The number of positivity constraints implied by \eqref{eq:ut}
is equal to $m$, the number of ODEs being solved.  This number
may be very large, for instance if the system is a semi-discretization
of a PDE.  This makes solution of the LP very costly.
But in most cases, positivity is violated only for a very small subset $h \subseteq \{1, \dots, m\}$ of
the solution components.  We can solve a much less expensive LP
by replacing \eqref{eq:ut} with
\begin{equation}
u_i^n + \dt \sum_{j=0}^s F_{i,j}  b_j  \geq 0 \quad \forall i \in h \subseteq \{1,\dots,m \}.
\end{equation}
Of course, it must be checked that the solution of the resulting LP still
satisfies the full set of constraints \eqref{eq:ut}.  In practice, we
have found the following approach to be effective.
First, set
$$h_0 = \{ i \in \{1,\dots,m \} \ | \  u_i^{n+1}  < 0 \}.$$
Solve the LP and let $\ut^{n+1}$ denote the new solution.
If $\ut^{n+1}$ satisfies \eqref{eq:ut}, accept this as the new
solution; otherwise, repeatedly take
$$h_{a+1} = \{ i \in \{1,\dots,m \}|  \ut_i^{n+1}  < 0 \} \cup h_{a}$$
until $\ut^{n+1}$ is found to satisfy \eqref{eq:ut}.
In the examples we have studied, this approach was found to always converge in
at most 2 iterations.

When enforcing a maximum value, the number of constraints can be reduced using the same technique. When enforcing both maximum and minimum values two separate sets of active constraints are used.
In this case it is important to update these sets simultaneously.

\subsection{Summary of the algorithm}

Our proposed method to solve a positive initial value problem \eqref{ode}
is summarized in Algorithm\,\ref{alg:Adaption}.

\begin{algorithm}[ht]
\begin{algorithmic}[1]
\State Initialize $n \leftarrow 0$, $u^n \leftarrow u_0$, $t \leftarrow 0$
\While{$t<t_{end}$}
\State Choose $\Delta t$ (fixed or via an adaptive stepsize control)
\State Calculate $F = (f_1, \dots, f_s)$ according to \eqref{eq:stagevalues} \label{line:calc_stge}
\State $u^{n+1} \leftarrow u^n + \Delta t Fb$, according to \eqref{eq:Combination}
\If {$u^{n+1} \geq 0$}
	\State  \textbf{GOTO} line\,\ref{line:update}
\Else
	\State{$\pt \leftarrow p_{start}$}
	\While{$\pt \geq p_{min}$}
		\State Solve LP \eqref{eq:free-adaptation}
		\If {LP is feasible}
			\State $\delta \leftarrow \| \Delta t F(\bt-b)\|$
			\If {$\delta < tol_\delta$}
				\State $u^{n+1} \leftarrow u^n+\Delta t F \tilde b$
				\State  \textbf{GOTO} line\,\ref{line:update}
			\EndIf
		\EndIf
		\State $\pt \leftarrow \pt - 1$
	\EndWhile
	\State Reduce $\dt$
	\State \textbf{GOTO} line\,\ref{line:calc_stge}
\EndIf
\State Estimate $error$ according to \eqref{eq:Err} \label{line:update}
\If {$error \leq tol_{error}$}
	\State $t \leftarrow t + \dt$, $n \leftarrow n+1$
\Else
	\State Reduce $\dt$
	\State \textbf{GOTO} line\,\ref{line:calc_stge}
\EndIf
\EndWhile
\end{algorithmic}
\caption{Pseudocode for the algorithm using a free adaption of weights.}
\label{alg:Adaption}
\end{algorithm}

\section{Properties of adaptive RKMs}

In the previous sections an algorithm for choosing positivity preserving weights $\bt$ has been presented.
In the next section properties of the adaptive RKMs are discussed.

\subsection{Choice of baseline method} \label{sec:integration}
An important property of the baseline method is the existence of embedded methods and the degrees of freedom for the weights $\bt$.
As noted in Section~\ref{sec:OrderCond} the number of stages has to be higher than the order.
It is natural to use explicit and diagonally implicit methods, both for their efficiency and because
the order need not be reduced as much in order to satisfy the condition $\pt<s$.
For a given method and reduced order $\pt$, the number of degrees of freedom for the choice of the new weights
is given by $s-\mathrm{rank}(Q_{\pt})$. 
The resulting number of degrees of freedom is shown in Table\,\ref{table:DOF_exp}
for some explicit methods and in Table\,\ref{table:DOF_imp} for several implicit
methods. The backward Euler extrapolation methods use the harmonic sequence
as described in \cite[Section~II.9]{hairer_solving_1993} and
\cite[Section~IV.9]{hairer_solving_1996}.

\begin{table}[h!]
\centering    
  \begin{tabular*}{\linewidth}{@{\extracolsep{\fill}}lr*6c@{}}
    \toprule
    Method & $s$ & \multicolumn{6}{c}{Order $\tilde p$} \\
    & & 1 & 2 & 3 & 4 & 5 & 6 \\
    \midrule
    Classical RK4 \cite{kutta1901beitrag} & 4 & 3 & 2 & 0 & 0 & --- & --- \\
    SSPRK(10,4) \cite{ketcheson2008highly} & 10&9&8&6&4& --- & ---\\
    Cash--Karp RK5(4)6 \cite{cash1990variable} & 6&5&4&2&1&0& --- \\
    Dormand--Prince RK5(4)7 \cite{prince1981high}& 7&6&5&3&1&0& --- \\
    \bottomrule
  \end{tabular*}
  \caption{Degrees of freedom for the choice of the weights for some explicit methods.} 
  \label{table:DOF_exp}
\end{table}

\begin{table}[h!]
\centering   
   \begin{tabular*}{\linewidth}{@{\extracolsep{\fill}}lr*6c@{}}
    \toprule
    Method & $s$ & \multicolumn{6}{c}{Order $\tilde p$} \\
    & & 1 & 2 & 3 & 4 & 5 & 6 \\
    \midrule
    Backward Euler& 1&0& --- & --- & --- & --- & ---  \\
    Lobatto~IIIC4 \cite{chipman1971stable} & 4&3&2&1&0&0&0 \\
    Radau~IIA3 \cite{ehle1969pade} & 3&2&1&0&0&0& ---  \\
    SDIRK(5,4) \cite[eq. (6.18)]{hairer_solving_1996}& 5&4&3&1&0& --- & ---  \\
    TR-BDF2 \cite{bank1985transient} & 3&2&1& --- & --- & --- & ---  \\
    Extrapolation BE~2 \cite[Sec.~II.9]{hairer_solving_1993} & 3&2&1& --- & --- & --- & ---  \\
    Extrapolation BE~3 \cite[Sec.~II.9]{hairer_solving_1993} & 6&5&4&2& --- & --- & ---  \\
    Extrapolation BE~4 \cite[Sec.~II.9]{hairer_solving_1993} & 10&9&8&6&3& --- & ---  \\
    \bottomrule
  \end{tabular*}
  \caption{Degrees of freedom for the choice of the weights for some implicit methods.} 
  \label{table:DOF_imp}
\end{table}

For explicit methods with the number of stages equal to the order of the
method, the order must be reduced in order to allow any freedom in the weights.
If the classical RK4 method is used the order has to be reduced more because
the RK4 method does not have embedded methods of order~3.
In contrast to this, some methods with $s > p$ admit changes to the weights without reducing the order.
An example of this is SSPRK(10,4), that has 4 degrees of freedom for $\pt = p$.
For Cash--Karp RK5 and Dormand--Prince RK5, even though the number of stages is higher than the order,
the order must be reduced in order to allow any modification of the weights.

Regarding implicit methods,
we can see that the fully implicit methods Lobatto~IIIC4 and Radau~IIA3 require a drastic reduction of the order, as expected.
The diagonally implicit SDIRK(5,4) method only requires an order reduction of one to get one degree of freedom for the weights.
The TR-BDF2 method even allows adaptations without reducing the order.
The backward Euler extrapolation methods also exhibit degrees of freedom without a reduction of the order.

It is also desirable that the baseline method have a large stability region.

Note that for many diagonally implicit methods, the first stage is a scaled backward Euler step.
For such methods, by allowing the order to be reduced to one
we can guarantee the existence of a solution to the LP, since the backward Euler method
is unconditionally positive.  For explicit methods, reducing the order to one is guaranteed
to yield a solution of the LP only if the step size is small enough.

\subsection{Error detection and approximation}\label{sec:error}
Stability analysis for the proposed approach is very challenging, since in principle
a different method may be used at every step.  At the same time, as long as the exact
solution is positive, we expect that as the step size goes to zero, eventually no modification
of the weights will be required and the convergence of the unmodified method will be observed.
This holds true in the examples shown in Section \ref{sec:Numeric_Results}.  We are thus more concerned with
the behavior of the modified method outside the asymptotic convergence regime.

To approximate the error of a new step we propose the following approximation of the local error:
\begin{align}
err = \|u(t^{n+1})-\tilde u^{n+1}\| &= \|u(t^{n+1}) - (u^{n+1}+\dt F(\bt-b))\| \\
 &\leq \underbrace{\|u(t^{n+1})-u^{n+1}\|}_{\approx err_T}+\underbrace{\|\dt F(\bt-b)\|}_{= \delta}. \label{eq:Err}
\end{align}

The total error is split up in the truncation error and the perturbation $\delta$ using the triangle inequality.
The truncation error can be estimated using the standard error estimators $err_T = \| u^{n}_{b} - u^{n}_{\hat{b}} \|$.
After adapting the weights, the perturbation is calculated. If the perturbation is larger than the tolerance, the weights are rejected.
The two values are added to get an approximation of the total error $err = err_T + \delta$.
This type of error estimation is easy to implement because it can be easily
incorporated in an existing step size control and takes advantage of the
standard error approximation.

\subsection{Stability region}

Adapting the weights $b$ changes the RK method. Hence, the stability
function is altered and the region of absolute stability varies.

As an example, the stability regions of adapted RKMs are visualized in
Figure\,\ref{fig:stab}.
In Figure\,\ref{fig:stab_dp5} the Dormand--Prince RK5 method is freely adapted.
The weights are taken from the example in Section~\ref{sec:Ex_expl}.
In Figure\,\ref{fig:stab_ex3} the stability regions of the BE~3
extrapolation method and the embedded chain of three BE steps with
time step $\dt/3$ are
plotted. Additionally the stability regions of convex combinations
of these two methods are shown.

\begin{figure}
     \centering
     \begin{subfigure}[b]{0.45\textwidth}
         \centering
         \includegraphics[width=\textwidth]{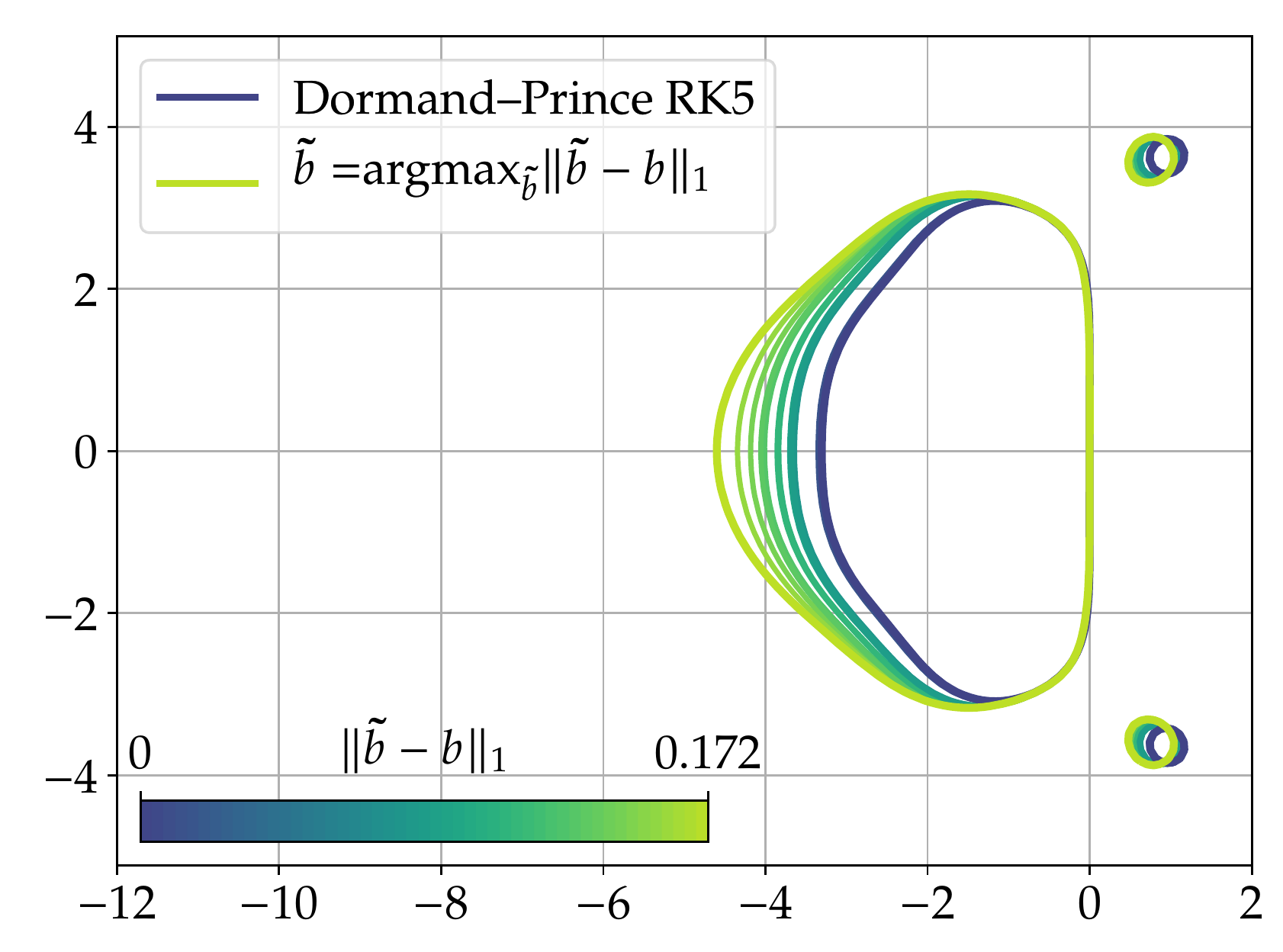}
         \caption{Dormand--Prince RK5 with free adaptation
                   of the weights $\bt$ as in the example in Figure~\ref{fig:weights_AdDe}.}
         \label{fig:stab_dp5}
     \end{subfigure}
     \hfill
     \begin{subfigure}[b]{0.45\textwidth}
         \centering
         \includegraphics[width=\textwidth]{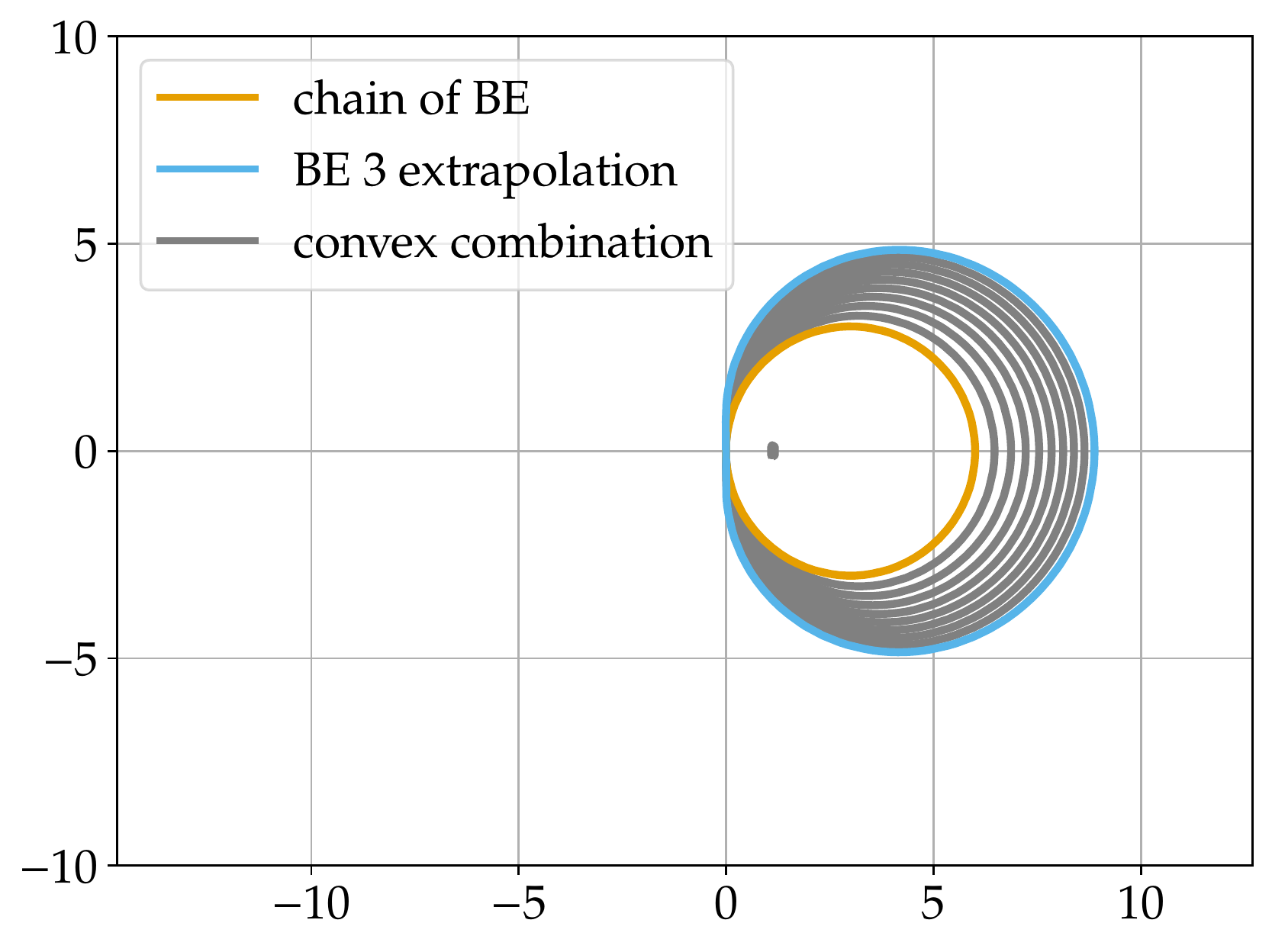}
         \caption{BE~3 extrapolation method and embedded chain
                  of three BE steps with convex combinations of both
                  weights.}
         \label{fig:stab_ex3}
     \end{subfigure}
        \caption{Change of stability region for the free adaptation and
                 the convex adaptation of the weights.
                 Regions with $|R(z)| \leq 1$ are hatched for the original methods.}
        \label{fig:stab}
\end{figure}

Let the stability function be denoted by $R(z): \CN \to \CN$.
Since we intend to vary the weights, we view $R(z)$
as a function parameterized by the weight vector $b$:
\begin{equation}
R_b(z) = 1 + zb^T(I - zA)^{-1}e,
\end{equation}
where $e = (1, \dots, 1)^T \in \R^s$.
The stability function is an affine function of the weights.

\subsubsection{Stability of convex adaptation}

If the new weights are chosen by convex adaptation of given weights,
it is easy to prove some properties of the stability region.
\begin{theorem} \label{thm:stability-region}
  The stability region of a Runge--Kutta method $(A,\bt)$ where
  $\bt = \sum_{i} g_i b^i$ is a convex combination of
  $b^1, \dots, b^m \in \R^s$ (i.e.\ $g_i \in [0,1]$, $\sum_i g_i = 1$),
  contains the intersection of the stability regions of
  the methods $(A,b^1), \dots, (A,b^m)$.
\end{theorem}
\begin{proof}
  Since the stability function is an affine-linear function of the
  weights, $R_{\sum_i g_i b^i}(z) = \sum_i g_i R_{b^i}(z)$. Hence,
  if $z$ is in the stability region of all methods $(A,b^1), \dots, (A,b^m)$,
  \begin{equation}
    | R_{\sum_i g_i b^i}(z) |
    \le
    \sum_i g_i | R_{b^i}(z) |
    \le
    1.
  \end{equation}
\end{proof}

This result is particularly important for implicit methods.
If all the embedded methods used to construct the new weights
are A-stable, the resulting method is also A-stable.

\subsubsection{Stability of free adaptation}
If the weights are adapted freely, in general we have no result like Theorem \ref{thm:stability-region}.
Still, if the change in the weights is small then the resulting stability
function is by some measure similar to the stability function of the baseline
method.

\begin{lemma}
  The stability function $R_{\bt}$ of an adapted RK method
  satisfies
  \begin{equation}
    | R_{\bt}(z) |
    \le
    | R_{b}(z) |
    + \| \bt - b \|_1 \| z (\I - z A)^{-1} e \|_\infty.
  \end{equation}
\end{lemma}
\begin{proof}
  Compute
  \begin{equation}
  \label{eq:proof-of-lemma-free-adaptation}
  \begin{aligned}
    | R_{\bt}(z) |
    &\le
    | R_{b}(z) |
    + | R_{\bt}(z) - R_{b}(z) |
    =
    | R_{b}(z) |
    + | z (\bt - b)^T (\I - z A)^{-1} e|
    \\
    &\le
    | R_{b}(z) |
    + \| \bt - b \|_1 \| z (\I - z A)^{-1} e \|_\infty.
  \end{aligned}
  \end{equation}
\end{proof}
This result suggests that the objective $\min \| \bt - b \|_1 $ is an
appropriate choice to control the change of the region of absolute stability,
in particular for explicit methods for which
$\| z (\I - z A)^{-1} e \|_\infty$ can be bounded by a polynomial
in $|z|$.

\section{Results of numerical experiments}\label{sec:Numeric_Results}

The implementation of the algorithms described above and code to
reproduce the numerical examples reported here can be found in
\cite{nusslein2020positivityRepro}.
The methods are implemented in Python using NumPy/SciPy
\cite{virtanen2020scipy}, NodePy \cite{ketcheson2020NodePy}, and
Matplotlib \cite{hunter2007matplotlib} for the visualizations.
We have used MOSEK \cite{mosek} via CVXPY to solve the LPs
\cite{cvxpy, cvxpy_rewriting}.

The adaptive RKM can be used with ODEs that satisfy \eqref{eq:condition_ODE_pos}.
For problems where the exact solution is positive for certain $u_0$ but do not satisfy \eqref{eq:condition_ODE_pos} tests did not show promising results. Additionally, it is not certain whether the computed solutions would be reasonable.

\subsection{Non-stiff problem with fixed stepsize}\label{sec:Ex_expl}
First, adaptive RKMs based on explicit methods are tested on non-stiff problems with a fixed step size.
When used with explicit methods the cost of solving the LP is significant because the computation of the stage derivatives only requires $s$ evaluations of the right-hand side (RHS).
For most of the linear test problems tried, the explicit methods yield to positive results.
When increasing the step size, issues with stability occur before getting negative values.
An example for this is the ODE in Section~\ref{sec:example_lin}.
Some nonlinear RHS may require very small time steps to preserve positivity.
For these, adapting the weights could be a possible way to solve them.
An example is the reaction equation solved in Section~\ref{sec:example_reac}.
Since the stage values are not guaranteed to be positive, the RHS must be
defined also
for negative values. If there is not a natural definition for negative values,
one can instead extend the function in a smooth way or simply replace negative
stage values by zero,  e.g.\ replace \texttt{sqrt(u)} by \texttt{sqrt(max(u, 0))}.

A test problem similar to \cite{shampine_non-negative_2005} is the PDE
\begin{equation}
\label{eq:advection-decay}
\begin{aligned}
  \frac{\partial u(t,x)}{\partial t}
  &=
  -a \frac{\partial u(t,x)}{\partial x} - K u(t,x),
  && x \in (0, 1), t \in (0,1),
  \\
  u(t,0) &= 1,
  && t \in (0,1],
  \\
  u(0,x) &= 0,
  && x \in [0,1],
\end{aligned}
\end{equation}
which consists of an advection part and an exponential decay.
The numerical approximation uses the method of lines and a first order
upwind finite difference semidiscretization with $N = 100$ points.
This leads to the positivity preserving ODE
\begin{equation}
\label{eq:advection-decay-semidiscrete}
\frac{\mathrm d}{\mathrm d t} u_i = \frac{a}{\Delta x} \left( u_{i-1} - u_i \right) - K u_{i}.
\end{equation}
The parameters are set to $a=1$ and $K=1$.
We use the Dormand--Prince RK5 method and adapt the weights using the free adaptation.
In Figure\,\ref{fig:sol_AdDe} the results for $\Delta t = 0.015$ are plotted for different  values of time $t$. We can see that the solution approaches an exponential function with $t \rightarrow \infty$.
In Figure\,\ref{fig:weights_AdDe} the used weights are plotted. For $t\leq 0.375$ the weights are altered. For $t > 0.375$ the original weights lead to a positive solution.
\begin{figure}
\centering
\begin{subfigure}[t]{0.45\textwidth}
\centering
\includegraphics[width=1\textwidth]{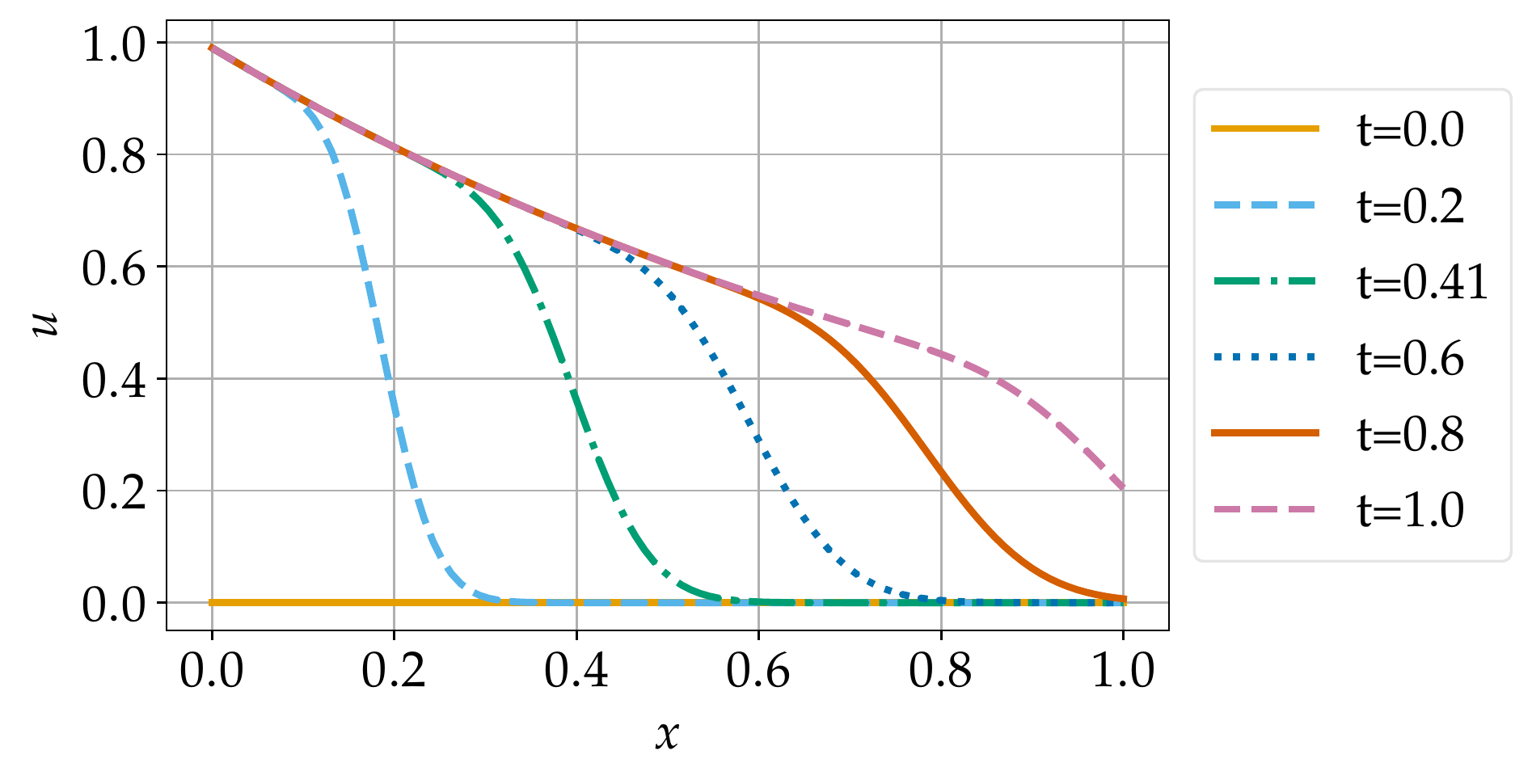}
\caption{Numerical solution at different times.}
\label{fig:sol_AdDe}
\end{subfigure}%
~
\begin{subfigure}[t]{0.45\textwidth}
\centering
\includegraphics[width=1\textwidth]{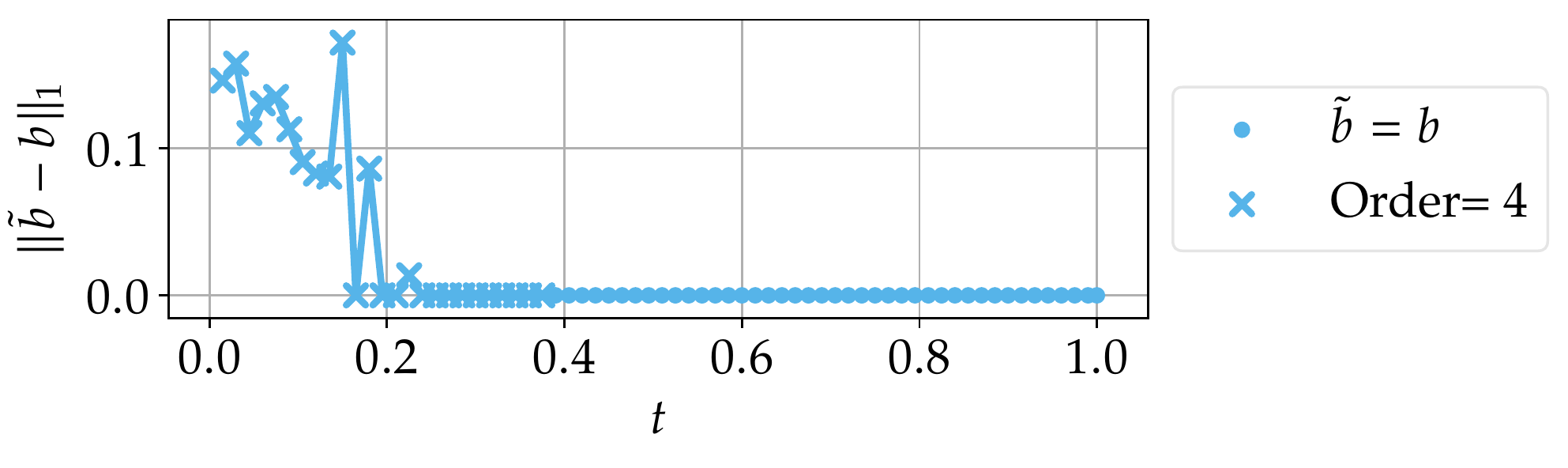}\\
\caption{Adaptation of the weights.}
\label{fig:weights_AdDe}
\end{subfigure}
\caption{Numerical results for the advection decay problem \eqref{eq:advection-decay}.}
\end{figure}

Next, different time steps are used.
For $\dt \leq 0.0082$ the unaltered method leads to positive solutions. For a larger $\dt$ the original method leads to negative values and the weights are altered. For $\dt >0.016$ the baseline method is no longer stable.
For the ODE, the reference solution can be computed using the matrix exponential.
In Figure\,\ref{fig:conv_expl} the convergence for $t=0.5$ is plotted for the altered and unaltered method.

\begin{figure}[ht]
\centering
\includegraphics[width=0.8\textwidth]{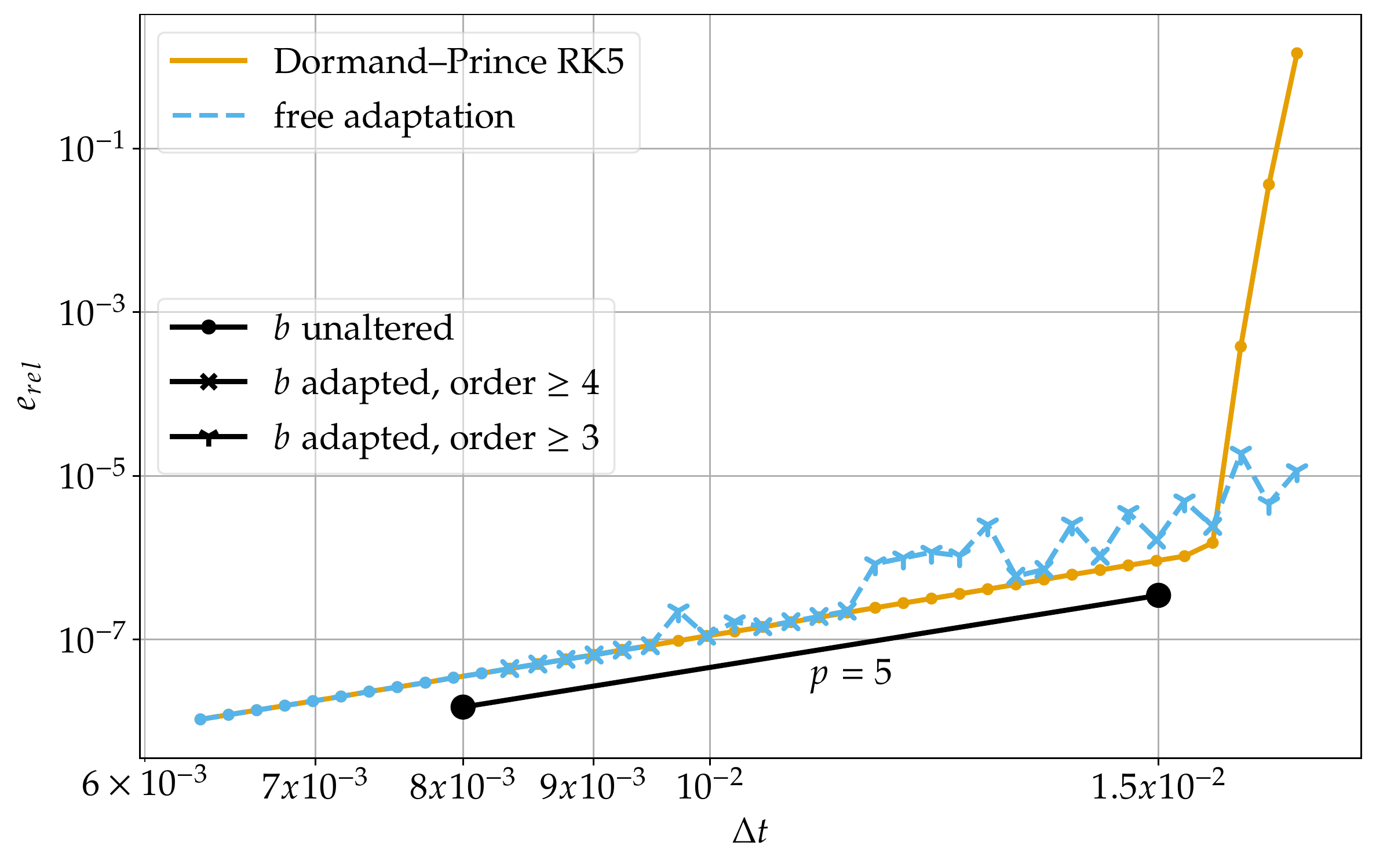}
\caption{Convergence results of the adapted Dormand--Prince RK5 for the advection decay problem \eqref{eq:advection-decay}.}
\label{fig:conv_expl}
\end{figure}

The unaltered method has the order $p=5$.
The values marked with a cross denote the numerical experiments that required an adaption of the weights.
Even though the order is reduced, most errors are still close to the error of the unaltered method.

It is natural to ask whether the adaptation of the weights through the algorithm proposed
here is more efficient than simply using a smaller step size.  For this problem that can
be discretized explicitly with a right-hand-side that is relatively cheap to evaluate,
using a smaller step size is
generally more efficient, at least with the current un-optimized implementation of the LP
solution.  For the problem considered in the next section, where an implicit
integrator is used, it is more efficient to maintain positivity with our proposed
approach instead of reducing the step size.
Adaptation of the weights could be made even more efficient with
an optimized implementation of the LP setup and solve; this is the subject of future work.

\subsection{Stiff problem with fixed step size}
Next, adaptive RKMs based on implicit methods are tested on stiff problems.
Implicit methods are an advantageous choice for a couple of reasons.
Firstly, the cost of solving the LP is relatively small compared to the cost of solving the stage equations.
Secondly, the time step is not limited by the stability of the method.
Therefore, it is possible to use larger time steps that are more likely to lead to negative values.

A very interesting class of methods are the implicit extrapolation methods. These allow changes of the weights without a reduction of the order, as discussed in Section\,\ref{sec:integration}.
Moreover, all stage values are computed using the BE method. Hence, all intermediate stages are positive.
Furthermore, an embedded BE step is included. This ensures that a positive solution always exists, even if it is of first order.

We test the proposed adaptation algorithm on the diffusion equation
\begin{equation}
\label{eq:diffusion}
\frac{\partial }{\partial t} u = D \frac{\partial^2}{\partial x^2} u
\end{equation}
with homogeneous Dirichlet boundary conditions on the domain $x = [-0.5,0.5]$ with $N=100$ points. The equation is semidiscretized using the 3-point-scheme
\begin{equation}
\frac{\mathrm d}{\mathrm d t} u_i = \frac{d}{\Delta x^2} \left( u_{i-1} - 2u_i + u_{i+1} \right).
\end{equation}
As initial condition $u^0 = (0,\cdots,0,1,0,\cdots,0)^T$ is used. The diffusion coefficient is $D=1$.

The ODE is solved using the BE~3 extrapolation method.
For large $\dt$ the method computes negative values for $u^1$.
These can be corrected by adapting the weights.
The solutions are computed using the free adaptation and convex adaptation for $\dt = \num{1e-3}$.
The results for the free adaptation are plotted in Figure\,\ref{fig:sol_Diff_a} and the corresponding change of the weights
is shown in Figure\,\ref{fig:weights_Diff_a}.
The original solution for the first step is negative. Therefore, the weights have to be changed.
If we take a look at the solution after the first time step at $t=0.001$ we can see that at $x=0$ the solution is smaller than the solution at the surrounding points.
This is not physical.
The next time steps lead to physical solutions again.
To prevent this glitch from happening we choose the weights based on a convex adaptation.
A first order embedded method is added. The solution is shown in Figure\,\ref{fig:sol_Diff_c} and the weights are visualized in Figure\,\ref{fig:weights_Diff_c}.
The weights for the first step are altered again.
The weights obtained by the convex adaptation are different from the weights obtained
by taking the free adaptation.
The solution for $t=0.001$ computed with the convex adaptation is physical.
For both approaches, the remaining steps can be computed with the standard weights.
\begin{figure}
\centering
\begin{subfigure}[b]{0.45\textwidth}
\centering
\includegraphics[width=1\textwidth]{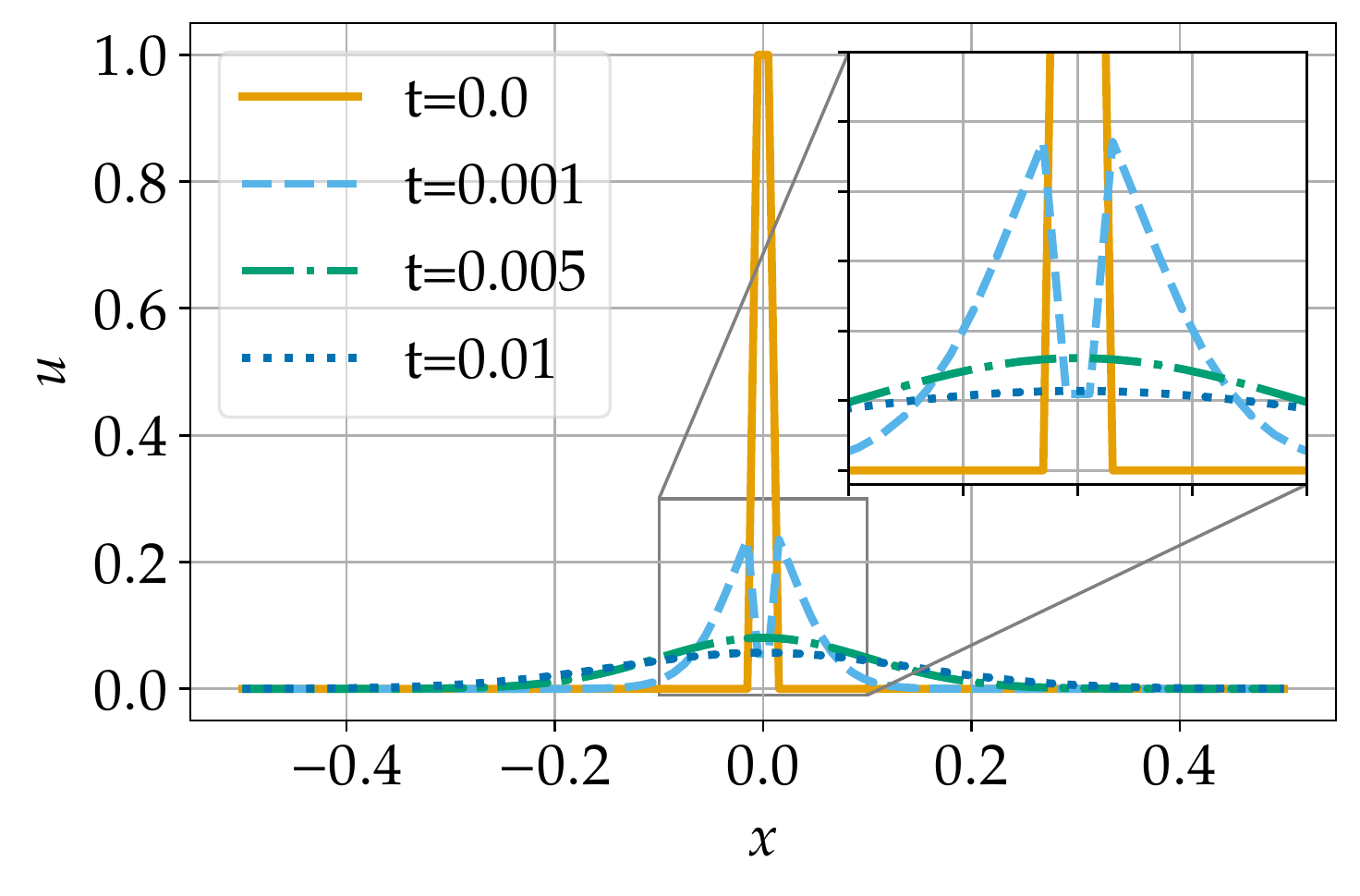}
\caption{Solution using the free adaptation.}
\label{fig:sol_Diff_a}
\end{subfigure}
\begin{subfigure}[b]{0.45\textwidth}
\centering
\includegraphics[width=1\textwidth]{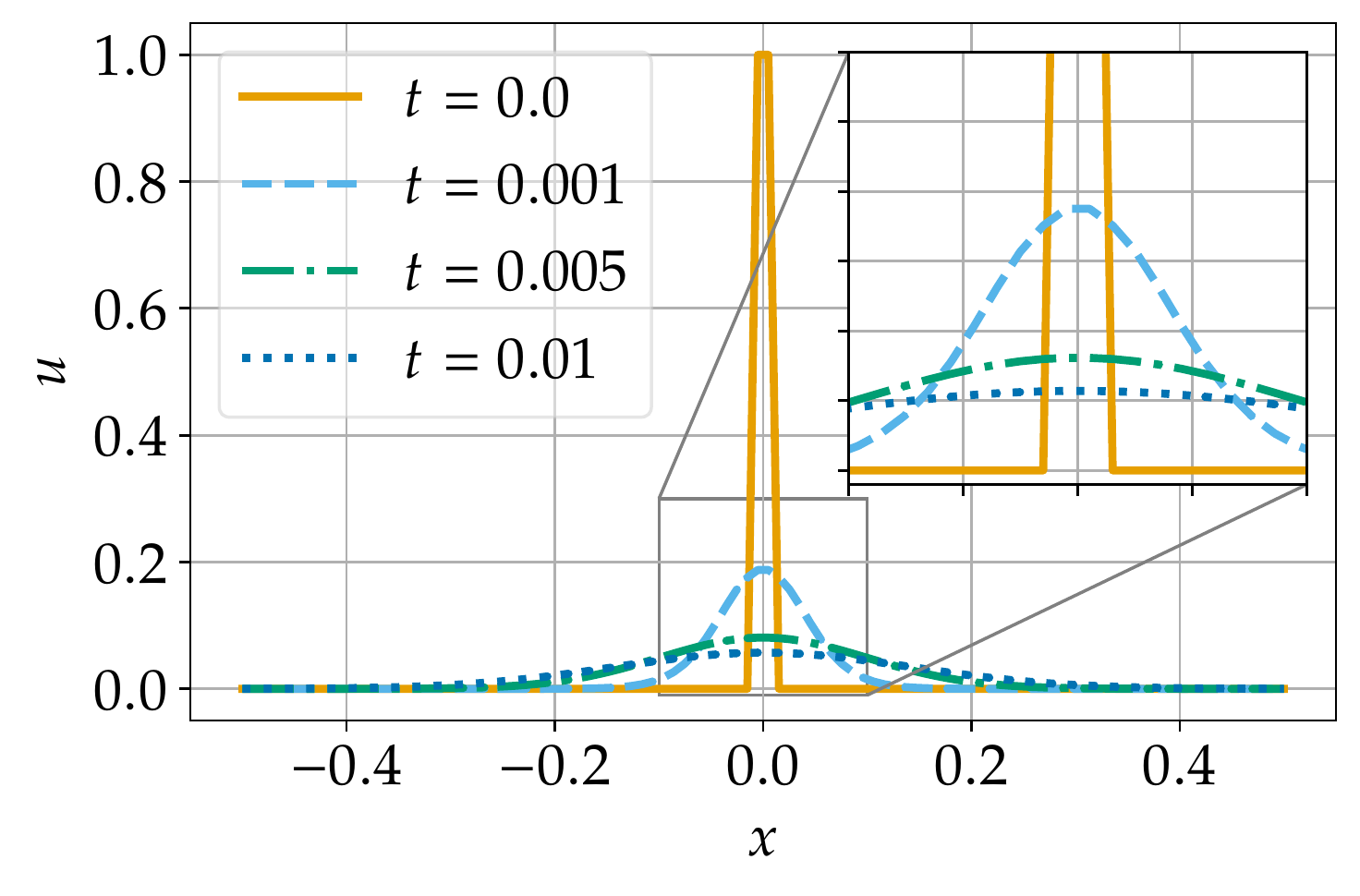}
\caption{Solution using the convex adaptation.}
\label{fig:sol_Diff_c}
\end{subfigure}
\\
\begin{subfigure}[b]{0.45\textwidth}
\centering
\includegraphics[width=1\textwidth]{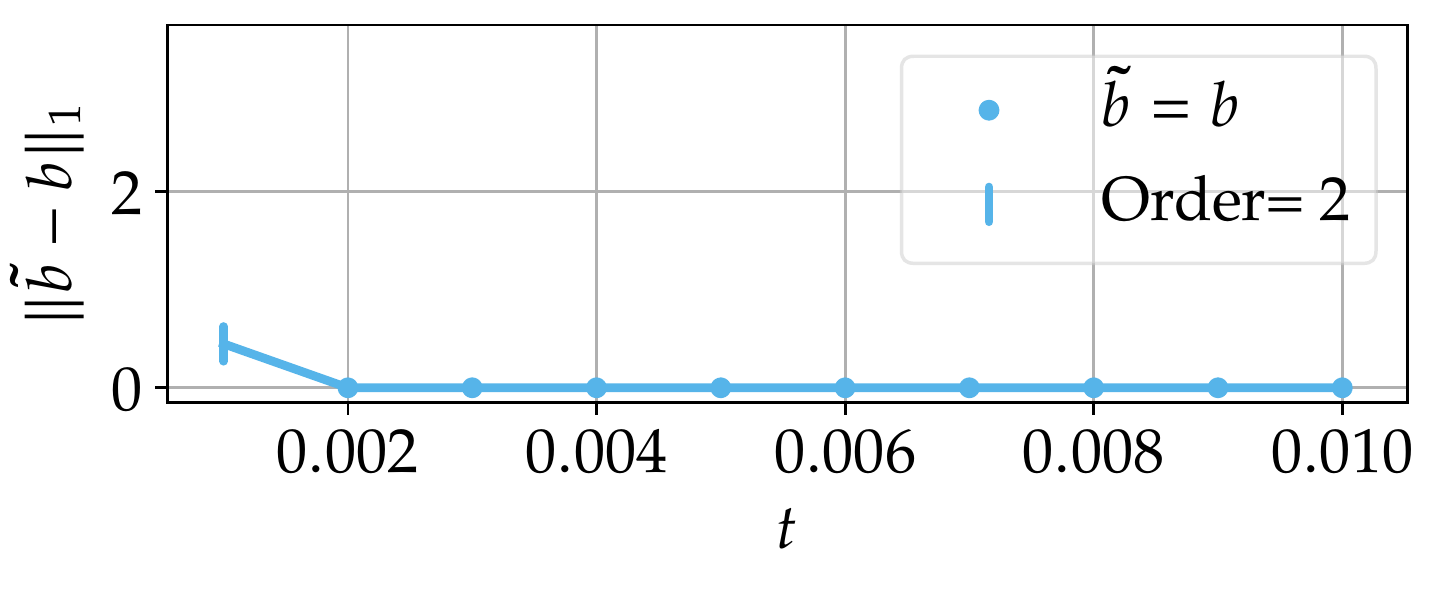}
\caption{Change of weights for free adaptation.}
\label{fig:weights_Diff_a}
\end{subfigure}
\begin{subfigure}[b]{0.45\textwidth}
\centering
\includegraphics[width=1\textwidth]{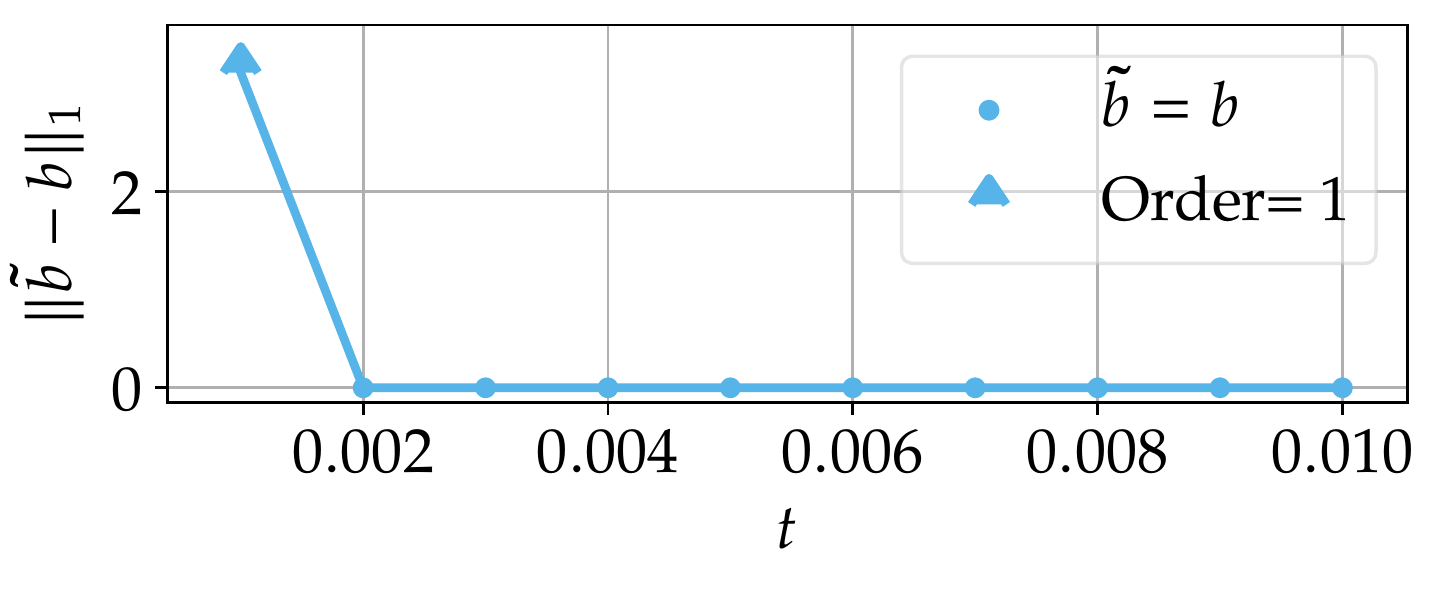}
\caption{Change of weights for convex adaptation.}
\label{fig:weights_Diff_c}
\end{subfigure}
\caption{Numerical results for the diffusion problem \eqref{eq:diffusion} and the
         adapted BE~3 extrapolation method.}
\end{figure}

In Figure\,\ref{fig:conv_impl}, the convergence is shown for the unaltered BE~3 extrapolation method (potentially resulting in negative values), the adaptive method with free adaptation, and the adapted method using convex adaptation.
Additionally, results for the BE method are plotted. It is only of 1st order but preserves positivity for all $\dt$.
For $\dt < \num{3e-5} $ the standard weights yield to a positive result. For larger $\dt$ the weights have to be adapted to ensure positivity.
The free adaptation results in similar convergence properties as the original method.
This can be expected, because the adapted method is still of 3rd order.
The convex adaptation yields larger errors than the free adaptation but leads to physical solutions for all time-steps.
This is no surprise because the adapted RKM used for the first step is only of first order.
But the adaptive method still outperforms the BE, even when accounting for the higher cost per step.

\begin{figure}[ht]
\centering
\includegraphics[width=\textwidth]{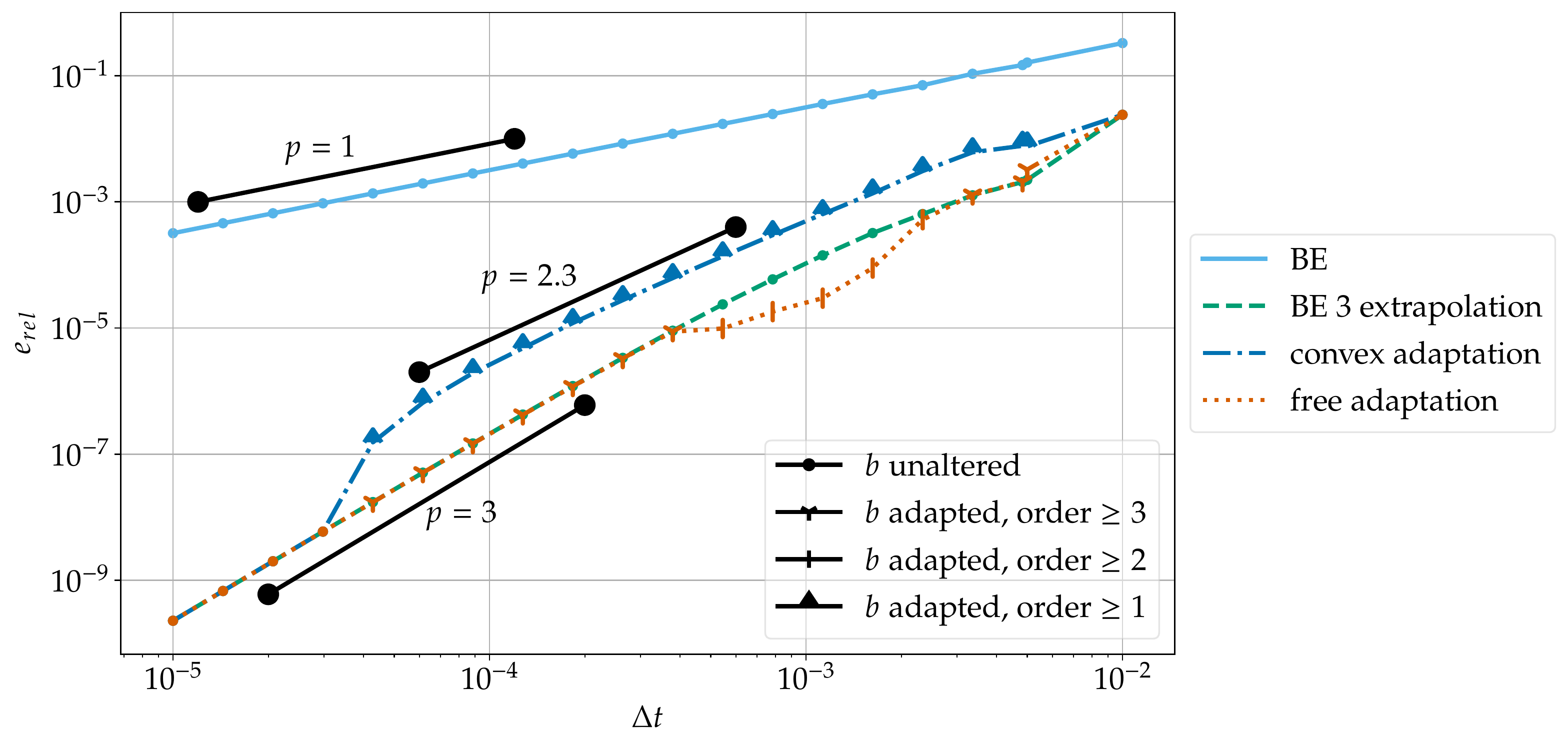}
\caption{Convergence result of baseline and adapted BE~3
         extrapolation methods for the diffusion problem \eqref{eq:diffusion}.}
\label{fig:conv_impl}
\end{figure}

\subsection{Stiff problem with adaptive step size}

Next we test the adaptive RKM on a more complex problem.
For this, we consider the advection-diffusion-production-destruction system  \cite{kopecz_comparison_2019}
\begin{subequations}
\label{eq:ADR}
\begin{align}
\frac{\partial u_1}{\partial t} &=-a \frac{\partial u_1}{\partial x} + d\frac{\partial^2 u_1}{\partial x ^2} + 0.01u_2 + 0.01 u_3 +0.003u_4 - \frac{u_1 u_2}{0.01+u_1}, \\
\frac{\partial u_2}{\partial t} &=-a \frac{\partial u_2}{\partial x} + d\frac{\partial^2 u_2}{\partial x ^2} + \frac{u_1u_2}{0.01+u_1} -0.01 u_2-0.5(1-\exp(-1.21 u_2^2)) u_3 -0.05 u_2, \\
\frac{\partial u_3}{\partial t} &=-a \frac{\partial u_3}{\partial x} + d\frac{\partial^2 u_3}{\partial x ^2} + 0.5(1-\exp(-1.21u_2^2)) u_3 - 0.01 u_3 -0.02 u_3, \\
\frac{\partial u_4}{\partial t} &=-a \frac{\partial u_4}{\partial x} + d\frac{\partial^2 u_4}{\partial x ^2} + 0.05 u_2 + 0.02 u_3 - 0.003u_4,
\end{align}
\end{subequations}
with parameters $a=\num{1e-2} $ and $ d=\num{1e-6}$.
The PDE is simulated on the domain $x = [0,1]$ with $N=100$ points and periodic boundary conditions.
The advection part is semidiscretized using a first order upwind scheme and the diffusion part is semidiscretized using a central 3-point-scheme. This leads to a positivity preserving system of ODEs which conserves the total mass $\sum u$.
The computation is done using the BE~3 extrapolation
method with free adaptation.
As step size control a PI-control from\,\cite{hairer_solving_1996} is used. The error was estimated using \eqref{eq:Err}. The tolerance was set to $Tol = 0.01$.
The final time is $t_{end} = 50$.

The simulation required 264 steps. Of these, 72 required an adaptation of the weights.
All adapted weights are still of 3rd order.
The solutions for $t=9,t=18,t=27$ and $t=50$ are plotted in Figure\,\ref{fig:Sol_ADP}.
In Figure\,\ref{fig:sol_ADP18} it can be seen that the reaction occurs in a small interfaces.
Outside of this regions quantities are close to zero. Therefore, it is very likely that negative values occur in the numerical approximation.
In Figure\,\ref{fig:sol_ADP27} it can be seen that at $T=27$ the two reaction interfaces merged. Afterwards the reaction stops and the behavior is mainly controlled by the advection and diffusion part.

In Figure\,\ref{fig:Stats_ADP} different values are plotted.
In the first subplot the step size is plotted. For $t<25$ the time steps are small. After $t = 30$ the step size increases, because the solution only evolves slowly afterwards.
In the second subplot the minimum of $u_1,u_2,u_3,u_4$ is plotted for all time steps that initially lead to negative values. This value is computed before and after adapting the weights.
We can see that relatively large negative values occurred at some time steps.
After the adaption of the weights, all values are close to $0$. Therefore, the adaption of weights successfully preserved positivity.
In the third subplot the approximated truncation error $err_T$ and the perturbation $\perturb$ are plotted.
We can see that $\perturb$ is of a similar magnitude as the truncation error. Therefore, the total error of the method is not increased drastically.
In the next subplots objective function is plotted.
We can see that the changes to the weights are only very small. The adapted RKM is still very close to the original RKM.
In the last subplot the deviation of the sum over $u_1,u_2,u_3,u_4$ from the initial sum is plotted.
The mass is conserved within roundoff error.

\begin{figure}
    \centering
    \begin{subfigure}[b]{0.49\textwidth}
        \includegraphics[width=\textwidth]{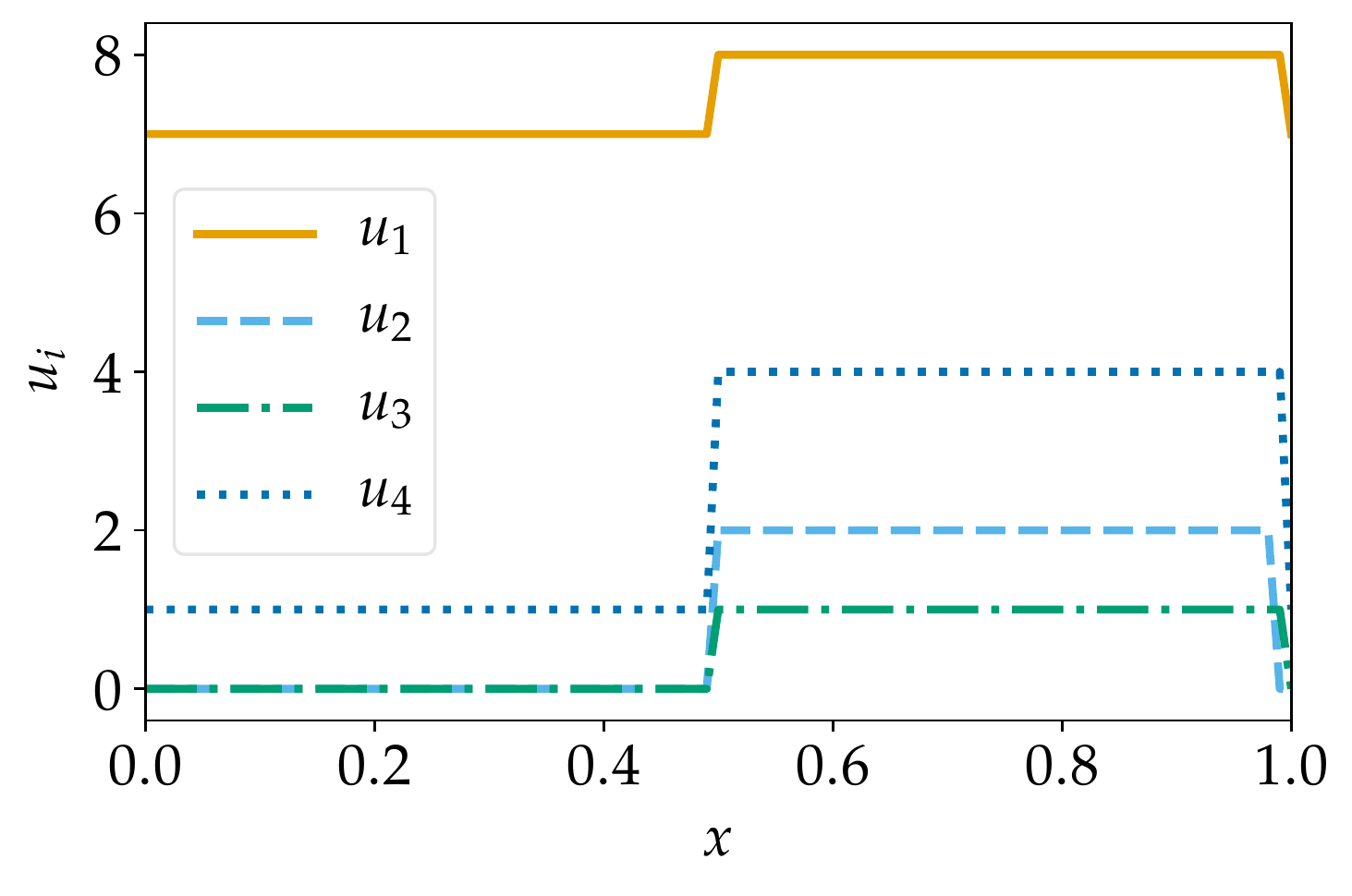}
        \caption{$t=0$.}
        \label{fig:sol_ADP00}
    \end{subfigure}
    \begin{subfigure}[b]{0.49\textwidth}
        \includegraphics[width=\textwidth]{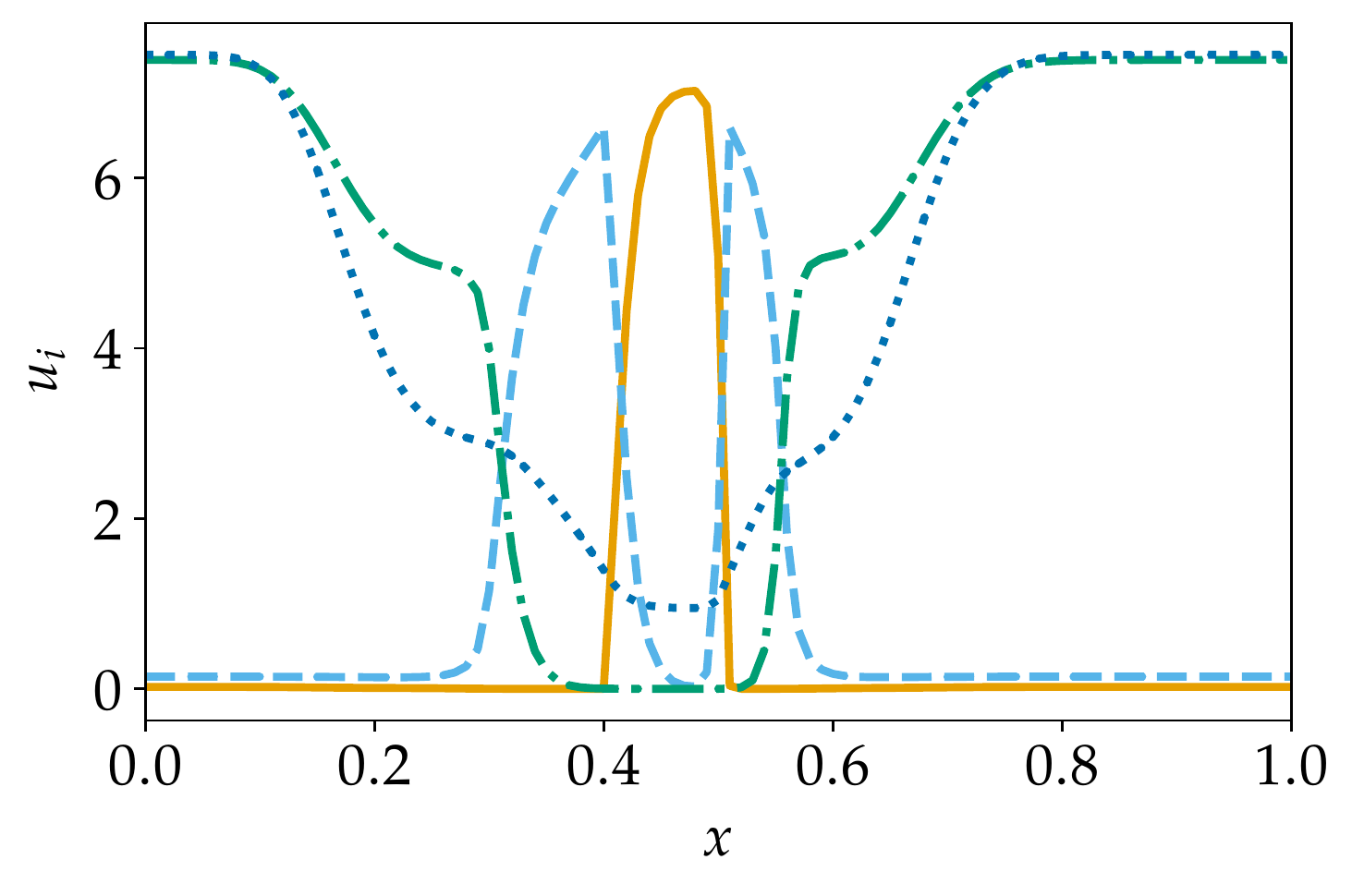}
        \caption{$t=18$.}
        \label{fig:sol_ADP18}
    \end{subfigure}

    \begin{subfigure}[b]{0.49\textwidth}
        \includegraphics[width=\textwidth]{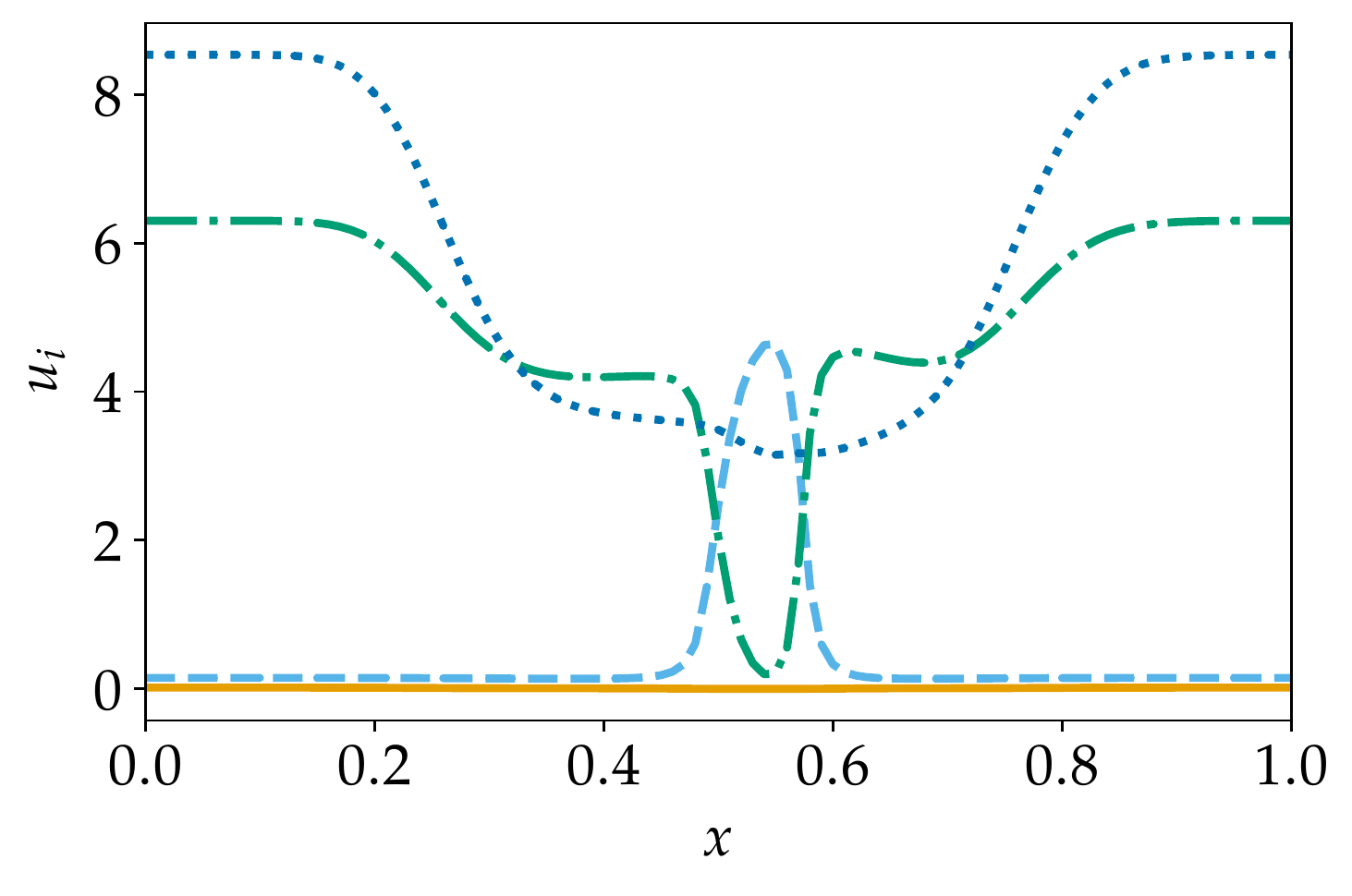}
        \caption{$t=27$.}
        \label{fig:sol_ADP27}
    \end{subfigure}
	\begin{subfigure}[b]{0.49\textwidth}
        \includegraphics[width=\textwidth]{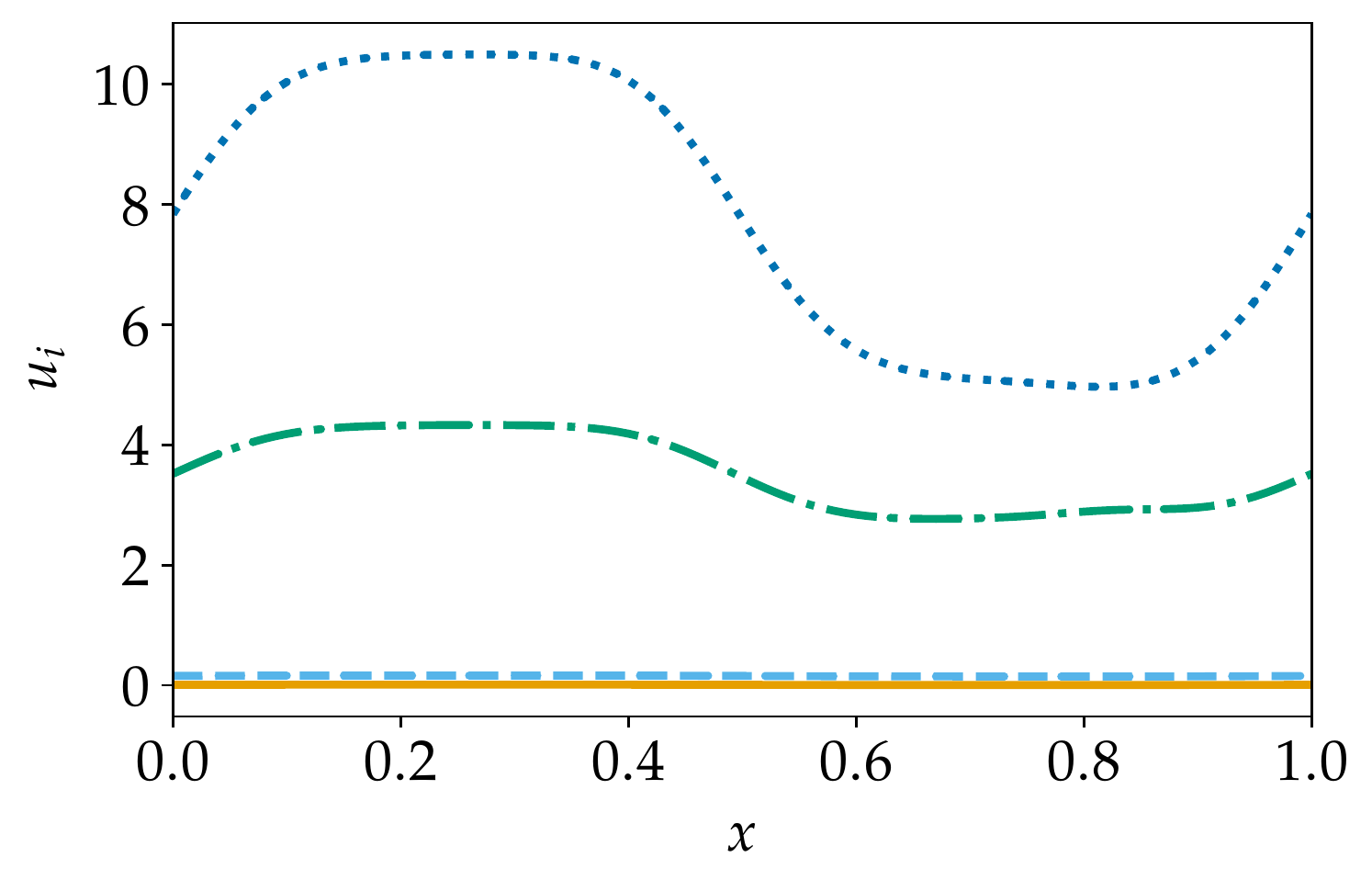}
        \caption{$t=50$.}
        \label{fig:sol_ADP50}
    \end{subfigure}
    \caption{Numerical solution of the advection-diffusion-reaction problem \eqref{eq:ADR} at different times.}\label{fig:Sol_ADP}
\end{figure}

\begin{figure}
\centering
\includegraphics[width=0.85\textwidth]{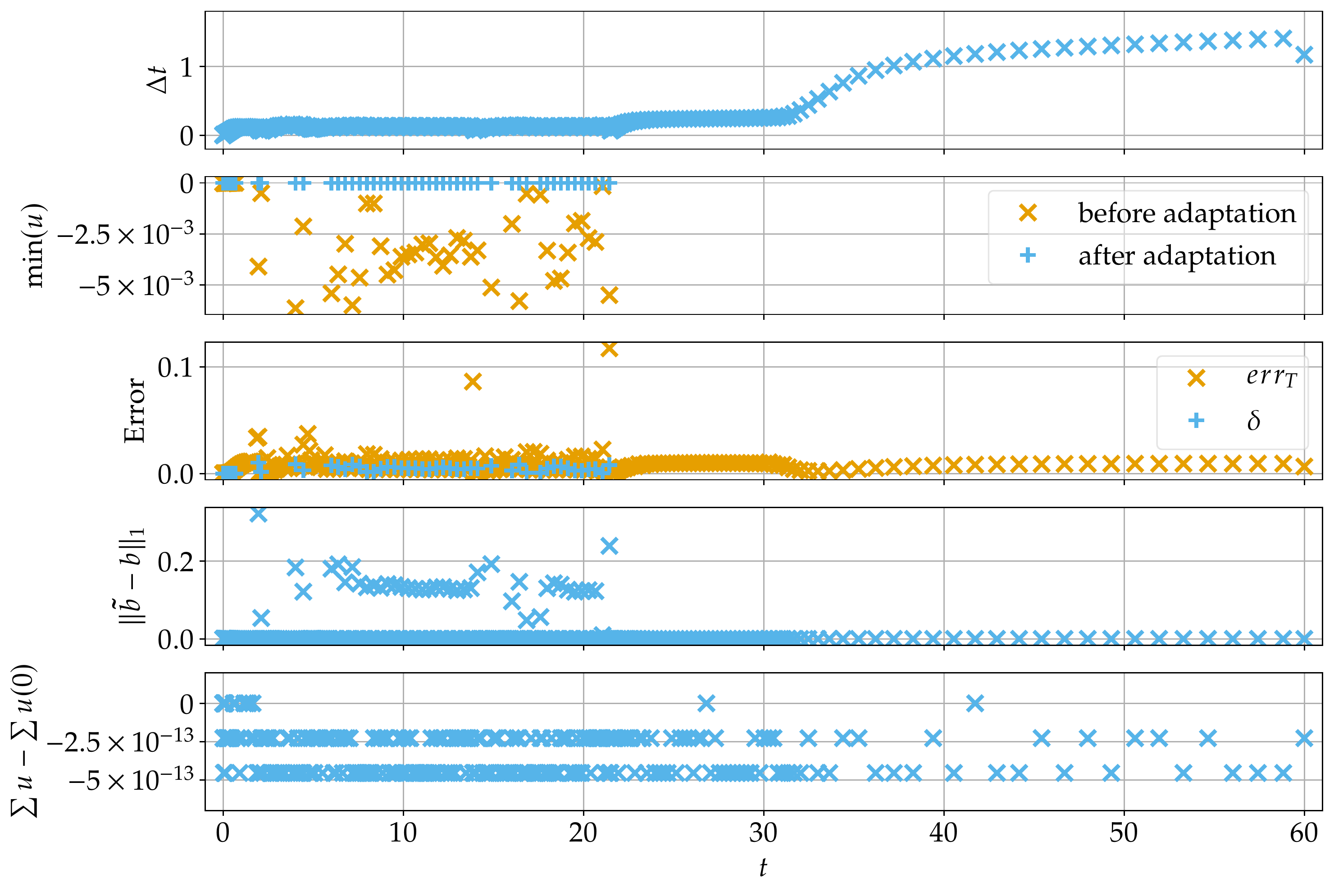}
\caption{Statistics of computation of the advection-diffusion-reaction problem \eqref{eq:ADR}.}
\label{fig:Stats_ADP}
\end{figure}

\subsection{Stiff problem with adaptive step size II}

Here, we consider the stratospheric reaction problem of
\cite{sandu2001positive}, which models the reactions of the substances
in the concentration vector $u = [O^{1D}, O, O_3, O_2, NO, NO_2]$.
This ODE has two linear invariants
\begin{align}
  m_O^T u &= \text{const}, & m_O &= [1,1,3,2,1,2]^T, \\
  m_N^T u &= \text{const}, & m_N &= [0,0,0,0,1,1]^T,
\end{align}
which describe the conservation of the total mass of oxygen and nitrogen,
respectively.

The reaction system
\begin{equation}
\label{eq:stratospheric}
\begin{aligned}
  \frac{\mathrm d}{\mathrm d t} O^{1D} &= r_5 - r_6 - r_7, \\
  \frac{\mathrm d}{\mathrm d t} O      &= 2 r_1 - r_2 + r_3 - r_4 + r_6 - r_9 + r_{10} - r_{11}, \\
  \frac{\mathrm d}{\mathrm d t} O_3    &= r_2 - r_3 - r_4 - r_5 - r_7 - r_8, \\
  \frac{\mathrm d}{\mathrm d t} O_2    &= -r_1 - r_2 + r_3 + 2 r_4 + r_5 + 2r_7 + r_8 + r_9, \\
  \frac{\mathrm d}{\mathrm d t} NO     &= -r_8 + r_9 + r_{10} - r_{11}, \\
  \frac{\mathrm d}{\mathrm d t} NO_2   &= r_8 - r_9 - r_{10} + r_{11},
\end{aligned}
\end{equation}
with time $t$ in seconds is given by the reaction rates
\begin{equation}
\begin{aligned}
  r_{1}  &= k_1\, O_2,          &  k_{1}  &= \num{2.643e-10} \sigma^3, &
  r_{2}  &= k_2\, O\, O_2,      &  k_{2}  &= \num{8.018e-17}, \\
  r_{3}  &= k_3\, O_3,          &  k_{3}  &= \num{6.120e-4} \sigma, &
  r_{4}  &= k_4\, O_3\, O,      &  k_{4}  &= \num{1.567e-15}, \\
  r_{5}  &= k_5\, O_3,          &  k_{5}  &= \num{1.070e-3} \sigma^2, &
  r_{6}  &= k_6\, M\, O^{1D},   &  k_{6}  &= \num{7.110e-11}, \\
  r_{7}  &= k_7\, O^{1D}\, O_3, &  k_{7}  &= \num{1.200e-10}, &
  r_{8}  &= k_8\, O_3\, NO,     &  k_{8}  &= \num{6.062e-15}, \\
  r_{9}  &= k_9\, NO_2\, O,     &  k_{9}  &= \num{1.069e-11}, &
  r_{10} &= k_{10}\, NO_2,      &  k_{10} &= \num{1.289e-2} \sigma, \\
  r_{11} &= k_{11}\, NO\, O,    &  k_{11} &= \num{1.0e-8},
\end{aligned}
\end{equation}
where $M  = \num{8.120e16}$ and
\begin{align}
    T &= (t/3600) \bmod 24,
    \quad T_r = 4.5,
    \quad T_s = 19.5 \\
    \sigma(T) &=
    \begin{cases}
    0.5+0.5 \cos\Bigl(\pi \Bigl|\frac{(2T-T_r-T_s)}{(T_s-T_r)}\Bigr| \frac{(2T-T_r-T_s)}{(T_s-T_r)} \Bigr)       & \quad \text{if } T_r \leq T \leq T_s,\\
    0  & \quad \text{otherwise}.
    \end{cases}
\end{align}

The initial conditions are
\begin{equation}
  u(t_0) = [
    \num{9.906e+1},
    \num{6.624e08},
    \num{5.326e11},
    \num{1.697e16},
    \num{4.000e6},
    \num{1.093e9}
  ]^T.
\end{equation}
The system was normalized internally such that $\forall n\colon u_n(t_0) = 1$
for the computation to achieve a suitable error estimation.
The system is solved in the time from $t_{0} = \SI{12}{h}$ to $t_{end} = \SI{84}{h}$
using the BE~3 extrapolation method with free adaptation and step size
control.

\begin{figure}
\centering
\includegraphics[width=0.9\textwidth]{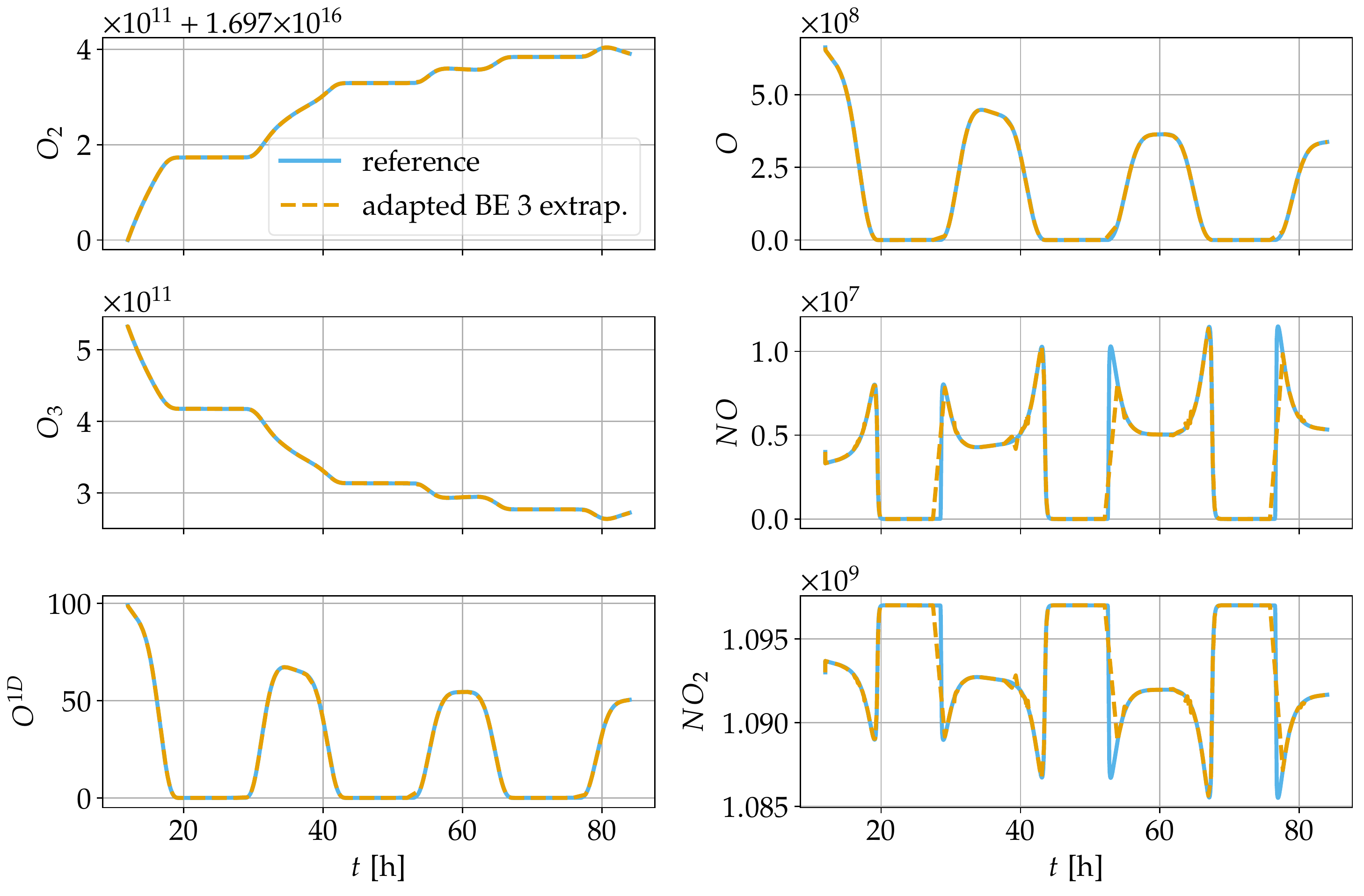}
\caption{Numerical approximation of stratospheric reaction system
         \eqref{eq:stratospheric}.}
\label{fig:Stratospheric_time}
\end{figure}

The results are shown in Figure~\ref{fig:Stratospheric_time}.
The adapted solution is close to the reference solution obtained
with the unadapted BE~3 extrapolation method and a higher accuracy.
For this solution, 249 steps were computed; two of these were rejected due to
a violation of the error bound.
More details are shown in Figure~\ref{fig:Stats_Strat}. The rejected steps are drawn with thick crosses.
The step size $\dt$ undergoes multiple sudden changes due to the explicit dependence
on time of the problem.
The minimum values before and after the adaptation are also shown for all steps where the initial values were negative.
This was only the case for some time intervals. 25 steps exhibited negative values.
Almost all of them were very close to zero and the adaptation did only show a small improvement.
The smallest value of the solution is $\min(u) = \num{-1.59e-11}$.
The used weights are also very close to the original weights, except of the steps that were rejected anyway due to a violation of the tolerance.
The change of the two linear invariants are shwon in the last two subplots. Both are preserved within roundoff error.

\begin{figure}
\centering
\includegraphics[width=0.85\textwidth]{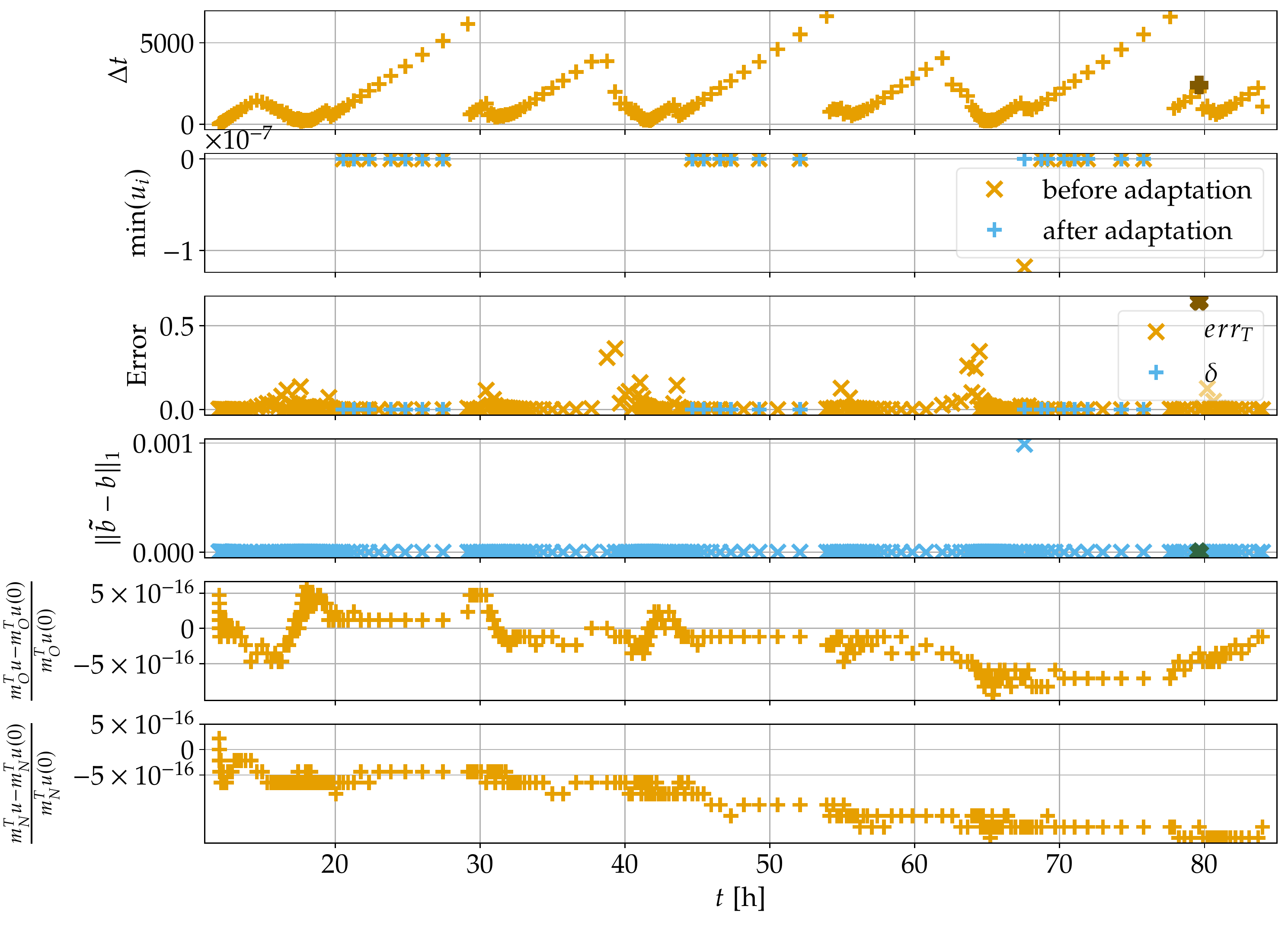}
\caption{Statistics of computation of the stratospheric reaction system
         \eqref{eq:stratospheric}.}
\label{fig:Stats_Strat}
\end{figure}

\section{Conclusion} \label{sec:conclusion}

It is possible to adapt the weights to enforce positivity for RKMs that are not positivity preserving.
One main limitation is that the resulting order has to be lower than the number of stages.
An error approximation for this method was given.
The region of absolute stability is altered by changing the weights. This effect can be predicted or controlled.
Used with explicit methods the positivity for some test problems could be recovered.
Because the time step size is limited by the stability it is only useful for a small interval of time steps.
The adaptive method is mainly interesting for diagonally implicit methods.
The times step size is not limited by stability.
Also, the cost of solving the LP is not a crucial factor.
If the negative values occurring are not too large, which can be expected for most computations, adapting the weights is a potential way to ensure positivity.

\printbibliography

\end{document}